\documentclass[11pt]{article}
\usepackage{amssymb, amsmath, amsthm, amsfonts, enumerate}
\usepackage{amsmath,setspace,scalefnt}
\usepackage[usenames,dvipsnames,svgnames,table]{xcolor}
\usepackage{graphicx,tikz,caption,subcaption}
\usetikzlibrary{patterns}
\usetikzlibrary{shapes}
\usetikzlibrary{decorations.pathreplacing}
\usetikzlibrary{decorations.pathmorphing}
\usetikzlibrary{positioning}
\usepackage{fullpage}
\usepackage{hyperref}
\usepackage{enumitem,verbatim}%,itemize
\usepackage{bbm,dsfont}

\definecolor{gray}{rgb}{0.25, 0.25, 0.25}

\newtheorem{theorem}{Theorem}[section]
\newtheorem{lemma}[theorem]{Lemma}

\theoremstyle{definition}

\theoremstyle{plain}
\newtheorem{claim}[theorem]{Claim}

\newtheorem{prob}[theorem]{Problem}
\newtheorem{prop}[theorem]{Proposition}
\theoremstyle{definition}

\theoremstyle{definition}

\theoremstyle{definition}

\theoremstyle{definition}
\newtheorem{defn}[theorem]{Definition}
\theoremstyle{definition}

\theoremstyle{definition}

\theoremstyle{definition}

\newenvironment{poc}{\begin{proof}[Proof of claim]}{\end{proof}}

\newcommand{\ep}{\varepsilon}
\newcommand{\eps}{\varepsilon}

\newcommand{\cP}{\mathcal{P}}

\renewcommand{\P}{\mathbb{P}}

\newcommand{\inte}{\mathfrak{int}}
\newcommand{\bd}{\mathfrak{bd}}
\newcommand{\tk}{\mathsf{TK}}

\newcommand{\bN}{\ensuremath{\mathbb{N}}}
\newcommand{\bP}{\ensuremath{\mathbb{P}}}
\newcommand{\bE}{\ensuremath{\mathbb{E}}}

\newcommand{\makenote}[2]
{

\smallskip

\noindent

\fbox{
\begin{minipage}{0.95\textwidth}

\def\temp{#1}
\ifx\temp\empty
\def\forlabel{$\bullet$}
\else
\def\forlabel{{\bfseries #1:}}
\fi

\begin{itemize}[label = \forlabel]
  #2
\end{itemize}
\end{minipage}
}

\smallskip
}

\definecolor{darkslateblue}{rgb}{0.28, 0.24, 0.55}
\definecolor{celestialblue}{rgb}{0.29, 0.59, 0.82}
\definecolor{cadmiumgreen}{rgb}{0.0, 0.42, 0.24}
\definecolor{darkpastelgreen}{rgb}{0.01, 0.75, 0.24}
\definecolor{deepcarmine}{rgb}{0.66, 0.13, 0.24}

\title{Disjoint isomorphic balanced clique subdivisions}
\author{
Irene Gil Fern\'andez
	\thanks{Mathematics Institute and DIMAP, University of Warwick, Coventry, CV4 7AL, UK. Email: {\texttt \{irene.gil-fernandez, Joseph.Hyde, O.Pikhurko, Zhuo.Wu\}@warwick.ac.uk}. J.H. was supported
by the UK Research and Innovation Future Leaders Fellowship MR/S016325/1. O.P. was supported by ERC Advanced Grant 101020255 and Leverhulme Research Project Grant RPG-2018-424.}
\and 
Joseph Hyde\footnotemark[1]
\and 
Hong Liu
\thanks{Extremal Combinatorics and Probability Group (ECOPRO), Institute for Basic Science (IBS), Daejeon, South Korea, Email: {\texttt hongliu@ibs.re.kr}. H.L. was supported by the Institute for Basic Science (IBS-R029-C4) and the UK Research and Innovation Future Leaders Fellowship MR/S016325/1.}
\and 
Oleg Pikhurko\footnotemark[1]
\and 
Zhuo Wu\footnotemark[1]
}

\begin{document}
\maketitle

%%%%%%%%%%%%%%%%%%%%%%%%%%%%%%%%%%%%%%%%%%%%%%%%
%%%%%%%%%%%%%%%%%%%%%%%%%%%%%%%%%%%%%%%%%%%%%%%%
%%%%%%%%%%%%%%%%%%%%%%%%%%%%%%%%%%%%%%%%%%%%%%%%
%%%%%%%%%%%%%%%%%%%%%%%%%%%%%%%%%%%%%%%%%%%%%%%%
%%%%%%%%%%%%%%%%%%%%%%%%%%%%%%%%%%%%%%%%%%%%%%%%

\begin{abstract}
    A thoroughly studied problem in Extremal Graph Theory is to find the best possible density condition in a host graph $G$ for guaranteeing the presence of a particular subgraph $H$ in $G$. One such classical result, due to Bollob\'{a}s and Thomason, and independently Koml\'{o}s and Szemer\'{e}di, states that average degree $O(k^2)$ guarantees the existence of a $K_k$-subdivision. We study two directions extending this result. 
    \begin{itemize}
        \item Verstra\"ete conjectured that the quadratic bound $O(k^2)$ would guarantee already two vertex-disjoint isomorphic copies of a $K_k$-subdivision.
        \item Thomassen conjectured that for each $k \in \mathbb{N}$ there is some $d = d(k)$ such that every graph with average degree at least $d$ contains a balanced subdivision of $K_k$, that is, a copy of $K_k$ where the edges are replaced by paths of equal length.  Recently, Liu and Montgomery confirmed Thomassen's conjecture, but the optimal bound on $d(k)$ remains open. 
    \end{itemize}
In this paper, we show that the quadratic bound $O(k^2)$ suffices to force a balanced $K_k$-subdivision. This gives the optimal bound on $d(k)$ needed in Thomassen's conjecture and implies the existence of $O(1)$ many vertex-disjoint isomorphic $K_k$-subdivisions, confirming Verstra\"ete's conjecture in a strong sense.
\end{abstract}

\section{Introduction}\label{sec-intro}

% A thoroughly studied problem in Extremal Graph Theory is to find the best possible density condition in a host graph $G$ that guarantees the presence of a particular subgraph $H$ in $G$. 
A \emph{subdivision} of a graph $H$ is obtained by replacing each edge of $H$ by a path, so that the new paths are internally vertex disjoint. This notion has played a central role in topological graph theory since the seminal result of Kuratowski in 1930 that a graph is planar if and only if it does not contain a subdivision of $K_5$ or $K_{3,3}$. In this paper, we are specifically interested in the optimal average degree for forcing particular subdivisions of a clique.    %the complete graph with five vertices or a subdivision of the complete bipartite graph with three vertices on each side.

 One of the first results in this direction was proved by Mader~\cite{Mader67}, who showed that there is some $d=d(k)$ such that every graph with average degree at least $d$ contains a subdivision of the complete graph~$K_k$. After some further results by Mader~\cite{Mader72}, it was proved by Bollob\'{a}s and Thomason~\cite{boltom}, and independently by Koml\'{o}s and Szemer\'{e}di~\cite{kstop1}, that we may take $d(k)=O(k^2)$. This is optimal: e.g., the complete balanced bipartite graph on $k^2/4$ vertices contains no subdivision of $K_\ell$ for $\ell\ge k$; indeed such a subdivision would require at least $\ell + \binom{\ell}{2} - (\frac{\ell}{2})^2 > \frac{k^2}{4}$ vertices, as at most $(\frac{\ell}{2})^2$ pairs of vertices can be embedded as adjacent pairs and each non-adjacent pair would require at least one additional vertex.

Twenty years ago, a strengthening of the result of Bollob\'{a}s-Thomason and Koml\'{o}s-Szemer\'{e}di was conjectured by Verstra\"ete~\cite{v}, who believed that the quadratic bound $O(k^2)$ suffices also to guarantee a pair of disjoint isomorphic subdivisions of $K_k$. This can be seen as a natural generalisation of the problem of finding disjoint cycles of the same length in a given graph. Such problem has received considerable attention since the work of Corradi and Hajnal~\cite{CH63}, who showed that for any positive integer $k$, any graph of order at least $3k$ and minimum degree at least $2k$ contains $k$ disjoint cycles. 
%Egawa~\cite{Ega96} proved that, in the above result, if in fact $|G|\geq 17k+o(k)$ then the cycles can be chosen to have the same length. Other variations are possible, for example Verstra\"ete~\cite{v} showed that for any positive integer~$k$, any bipartite graph $G$ with average degree at least $4k$ contains $k$ cycles with consecutive even lengths, and moreover, if $|G|\geq (16k^2)!$ and $d(G)=C k^2$, the same conclusion is true with the additional property that the cycles are disjoint. 

A different direction of extension was proposed by Thomassen~\cite{tsub} (see also \cite{tprob, tconf}), who conjectured that, for each $k\in \bN$, there is some $d=d(k)$ such that every graph with average degree at least $d$ contains a balanced subdivision of $K_k$. Here, a subdivision is \emph{balanced} if every added path is of the same length. In 2020, Liu and Montgomery~\cite{lmerdoshajnal} confirmed this conjecture, but it remains to determine optimal bounds for $d(k)$.
Clearly, we have by $d(k) = \Omega(k^2)$ by the same complete bipartite graph example above. Very recently, Wang~\cite{yw} proved that there exists $d$ such that every $n$-vertex graph with average degree at least $d$ contains a balanced subdivision of $K_r$, where $r=\Omega(d^{1/2}/{\log^{10}n})$.

%%%%%%%%%%%%%%%%%%%%%%%%

Our main result simultaneously settles the conjecture of Verstra\"ete and gives optimal bounds for $d(k)$ in Thomassen's conjecture.
It shows that the quadratic bound is optimal in forcing balanced disjoint isomorphic clique subdivisions. Indeed, simply notice that a balanced $K_{tk}$-subdivision contains $t$ disjoint isomorphic balanced $K_k$-subdivisions. 

\begin{theorem}\label{thm-main}
Every graph with average degree $d$ contains a balanced subdivision of a complete graph of order $\Omega(\sqrt{d})$.
\end{theorem}

Our approach uses a version of expander called sublinear expanders. For more recent applications of the theory of sublinear expanders, we refer the interested readers to~\cite{GFLiu22,GFKimKimLiu22,HKL20,CruxCycle,HHKL21,KLShS17,lmmader,lmerdoshajnal,LWY20}.

\paragraph{Organisation.} 
The rest of the paper is organised as follows.
Preliminaries are given in Section~\ref{sec-def}.
In Section~\ref{sec-def-komsze} we reduce to proving Theorem~\ref{thm-main} in the case when $G$ is a bipartite expander graph. 
Then in Section~\ref{sec-def-reduction}, we further reduce to proving Theorem~\ref{thm-main} when $n/K \geq d \geq \log^{800}(n)$ for some large constant~$K$ (stated as Theorem~\ref{main-theorem}). 
Section \ref{sec-proof-ideas} contains a proof sketch of Theorem~\ref{main-theorem}.
In Section \ref{sec-build-units}, we demonstrate the existence of a certain structure, called a \emph{unit}, in an expander graph. % Having these structures will be a crucial for building our absorbing structures (called `adjusters') later.
In Section \ref{sec-build-adj}, we use these units to build a certain absorbing structure introduced by Liu and Montgomery~\cite{lmerdoshajnal} called an \emph{adjuster}, which will allow us to `adjust' the length of a path between two points under certain conditions.
In Section \ref{sec-link-adj}, we will use these adjusters to prove Theorem~\ref{main-theorem}. Section~\ref{sec-conclude} contains concluding remarks.
 
\section{Definitions and some auxiliary results}\label{sec-def}

\paragraph{Notation.}
For sets $X$ and $Y$, we define $X \times Y := \{(x,y): x \in X, y \in Y\}$.
We often omit brackets when writing small sets, for example, abbreviating $\{x\}$ and $\{x,y\}$ to $x$ and $xy$, respectively. For $\ell \in \mathbb{N}$, we define $[\ell] := \{1, 2, \ldots, \ell\}$. We omit floor and ceiling signs when they are not essential, that is, we treat large numbers as integers. We sometimes write $(a,b,c,d)_{X} = (a',b',c',d')$ meaning that we choose constants $a,b,c,d$ in the statement of Result $X$ to be $a',b',c',d'$.

Let $G$ be a graph. We will denote the set of vertices of $G$ by $V(G)$ and the set of edges by $E(G)$, and define $|G| := |V(G)|$ and $e(G):=|E(G)|$.
Let $A \subseteq V(G)$.
We define $G[A]$ to be the \emph{subgraph of $G$ induced by $A$} with vertex set $A$ and edge set $\{xy \in E(G): x,y \in A\}$.
Further, we define $G - A:=G[V(G)\setminus A]$ to be the graph with vertex set $V(G)\setminus A$ and edge set $E(G)\setminus \{vw \in E(G): v \in A \;\text{or} \;w \in A\}$. 
  We define the \emph{external neighbourhood of $A$ in $G$} to be 
  $$
  N_G(A) := \{w \in V(G)\setminus A: \exists \; v \in A \; \mbox{such that}\; vw \in E(G)\}.
  $$
  For $k \in \mathbb{N}$ we define the \emph{ball of radius $k$ around $A$ in $G$}, denoted by $B_G^{k}(A)$, to be the set of vertices with graph distance at most $k$ from a vertex in~$A$.
For $v \in V(G)$, we define the \emph{degree of $v$ in $G$} to be $d_G(v) := |N_G(v)|$. Let 
$$
d(G):=\frac{1}{|G|}\, \sum_{v\in V(G)} d_G(v)
$$
denote the \emph{average degree} of~$G$.
For $A,B \subseteq V(G)$, we define $N_G(A,B) := N_G(A)\cap B$.
We define an \emph{$A,B$-path in $G$} to be a path which has one endpoint in $A$ and the other endpoint in $B$ and has no other vertices in $A$ or~$B$. 
For a subgraph $F \subseteq G$, we define $G\setminus F$ to be the graph with vertex set $V(G)$ and edge set $E(G)\setminus E(F)$. 
The \emph{length} of a path $P$ is $e(P)=|P|-1$,  the number of edges in it.
We say a collection of paths $\mathcal{P}$ is \emph{internally vertex disjoint} if for each pair of paths $P_1, P_2 \in \mathcal{P}$ 
the set of internal vertices of $P_1$ is disjoint from $V(P_2)$; in other words, if a vertex belongs to two different paths of $\mathcal{P}$ then it is an endpoint in both.
We note that sometimes, when it is clear from context, we drop the subscript $G$ from the above nomenclature.  For $h \in \mathbb{N}$, we define an \emph{$h$-star} to be the graph on $h+1$ vertices where one vertex has degree $h$ and all other vertices have degree one. 

For $\ell,t \in \bN$, we write $\tk^{(\ell)}_t$ for a balanced $K_t$-subdivision in which each edge is replaced by a path with $\ell$ internal vertices (i.e., a path of length $\ell+1$).

\subsection{Robust Koml\'{o}s-Szemer\'{e}di expansion}
\label{sec-def-komsze}

We use the following notion of expansion introduced by Haslegrave, Kim and Liu~\cite{HKL20}, which is essentially a robust form of the sublinear expansion property introduced by Koml\'{o}s and Szemer\'{e}di  in~\cite{kstop1, kstop2}. 
Informally speaking, this property states that even after removing a relatively small set of edges we can still guarantee sublinear expansion properties. % which is essentially the strongest type of expansion we can find in some subgraph of a graph 

\begin{defn}
   Let $\varepsilon_{1}>0$ and $k\in\bN$.
   A graph $G$ is an $(\varepsilon_{1},k)$-\textit{robust-expander} if for all $X\subseteq V(G)$ with $k/2\leq |X|\leq |G|/2$, and any subgraph $F\subseteq G$ with $e(F)\leq d(G)\cdot \eps(|X|)\cdot |X|$, we have
		$$
		|N_{G\setminus F}(X)| \geq \varepsilon(|X|)\cdot|X|,
		$$
		where
		\begin{equation}\label{eq:ep}
		\varepsilon(x)=\varepsilon\left(x, \varepsilon_{1}, k\right):=\left\{\begin{array}{cc}
			0 & \text { if } x<k / 5, \\
			\varepsilon_{1} / \log ^{2}(15 x / k) & \text { if } x \geq k / 5.
		\end{array}\right.
		\end{equation} 
\end{defn}

%Whenever our choices of $\ep_1$ and $k$ are clear, we omit them and write $\ep(x)$ for $\ep(x, \ep_1, k)$.

Observe that $\ep(x, \ep_1, k)$ decreases for $x \geq k/2$, but $\ep(x, \ep_1, k) \cdot x$ increases for $x \ge \frac{e^2 k}{15}$ (in particular, for $x \geq k/2$).

Importantly, Koml\'{o}s and Szemer\'{e}di~\cite{kstop2} proved that every graph $G$ contains an expander subgraph with comparable average degree to $G$, and Haslegrave, Kim and Liu~\cite[Lemma 3.2]{HKL20} proved the analogous result for robust-expanders. From now on we will refer to a robust-expander as an expander.

The following is a direct consequence of \cite[Lemma~3.2]{HKL20} and the well known facts that every graph $G$ has a bipartite subgraph~$H$ with $d(H) \geq d(G)/2$ and that every graph $H$ contains a subgraph of minimum degree at least $d(H)/2$.

\begin{lemma}\label{cor-expander}
    Let $\ep_1 \leq 1/1000$, $\ep_2 < 1/2$ and $d > 0$. Every graph $G$ with $d(G) \geq d$ has a bipartite $(\ep_1, \ep_2 d)$-expander subgraph $H$ with $\delta(H) \geq d/8$.
\end{lemma}

Note that the expansion subgraph $H \subseteq G$ found in Lemma~\ref{cor-expander} may be significantly smaller than $|G|$. Indeed, $G$ could be disjoint union of many copies of $H$.

A common use of the expansion condition is to connect %(usually large) 
two vertex sets together by a short path, which, as the following result shows, we can do even after removing a relatively smaller set of vertices. We will use the following version, which is a slight variation of Lemma 3.4 in \cite{lmerdoshajnal}. %(usually with paths of size comparatively much smaller than the size of the vertex sets being connected). 

\begin{lemma}\label{lem-mpath} 
	For each $0<\eps_1,\eps_2<1$, there exists $d_0=d_0(\eps_1,\eps_2)$ such that the following holds for each $n\geq d\geq d_0$ and $x\geq 1$.
		Let $G$ be an $n$-vertex $(\eps_1,\eps_2d)$-expander with $\delta(G)\geq d-1$. Let $A,B\subseteq V(G)$ with $|A|,|B|\geq x$, and  let $W\subseteq V(G)\setminus(A\cup B)$ satisfy $$|W|\leq \frac{\ep_1x}{4\log^2(\frac{15n}{\ep_2d})}.$$ Then, there is an $A,B$-path in $G-W$ with length at most $\frac{100}{\eps_1}\log^3\frac{15n}{\ep_2d}$.
\end{lemma}
    
\begin{proof}
    We first prove the following claim.
    \begin{claim}\label{claim-ball-expands}
        Set $m_0:=\frac{40}{\eps_1}\log^3\frac{15n}{\ep_2d}$. % and $\ell_0:=(\log\log n)^5$ 
        Let $C \subseteq V(G-W$) with $|C| \geq \ep_2d/2$. Let $\ep(y)=\ep(y,\ep_1, \ep_2 d)$ be the function defined in~\eqref{eq:ep}. If $|W|\leq \ep(|C|)\cdot|C|/4$, then $|B^{m_0}_{G-W}(C)|>n/2$.
    \end{claim}
    \begin{poc}
    For each $i\in\mathbb{N}$, denote $C_i:=B^{i}_{G-W}(C)$ and set $C_0 := C$. 
    Using the expansion property of $G$ (with $F$ as the empty graph), our assumption that $|W|\leq \ep(|C|)\cdot|C|/4$ and the fact that $\ep(y)\cdot y$ increases for $y \geq \ep_2d/2$, we have for any $i\le m_0+1$ that $|C_{i-1}| > n/2$ (and so the claim holds by $|C_{m_0}|\ge |C_{i-1}|$), or
    $$|N_G(C_{i-1})|\geq \ep(|C_{i-1}|)\cdot|C_{i-1}|\geq\ep(|C|)\cdot |C|\geq 4\cdot|W|.$$
    Assuming that we are not done yet, we have
    $$|N_{G-W}(C_{i-1})|\geq |N_{G}(C_{i-1})|-|W|\geq\frac{1}{2}\,|N_G(C_{i-1})|.$$
    Now, using the above two inequalities and that $\ep(y)$ decreases for $y \geq \ep_2 d/2$, we bound
    $$|N_{G-W}(C_{i-1})|\geq\frac{1}{2}\ep(|C_{i-1}|)\cdot|C_{i-1}|\geq \frac{1}{2}\,\ep(n)\cdot|C_{i-1}|.$$
    Then, since $|C| \geq x$ we have
    $$|C_{m_0}|\geq\Big(1+\frac{1}{2}\,\ep(n)\Big)^{m_0}\cdot x.$$
    Finally,  we can solve $\big(1+\frac{1}{2}\,\ep(n)\big)^r\cdot \ep_2 d>n/2$ for $r$, and using
    the inequality $\log(1+x)\geq (1+\frac{1}{x})^{-1}$, see that the desired inequality holds  for $r=m_0$.
    \end{poc}
 
\newcommand{\hide}[1]{}    
\hide{
        \textcolor{red}{\begin{align*}
             & \left(1+\frac{1}{2}\ep(n)\right)^r\cdot \ep_2d >  n/2 \\
            \iff       \     &  r\log\left(1+ \frac{1}{2}\ep(n)\right) >  \log\left(\frac{n}{2\ep_2d}\right) \\
            \Longleftarrow        \    & r\left(1+ \frac{2}{\ep_1}\log^2\left(\frac{15n}{\ep_2d}\right)\right)^{-1} >  \log\left(\frac{n}{2\ep_2d}\right) \\
            \Longleftarrow  \    & r >  \log\left(\frac{n}{2\ep_2d}\right) + \frac{2}{\ep_1}\log^2\left(\frac{15n}{\ep_2d}\right)\cdot \log\left(\frac{n}{2\ep_2d}\right) \\
            \Longleftarrow   \   & r >  \frac{1}{2\ep_2}\log\left(\frac{n}{d}\right) + \frac{15}{\ep_1\ep_2^2}\log^3\left(\frac{n}{d}\right) \\
            \Longleftarrow  \    & r >  \frac{20}{\ep_1\ep_2^2}\log^3\left(\frac{n}{d}\right).
    \end{align*}}
}     
    We have two cases. If $x \geq \ep_2d/2$, we have
    $$\ep(x)\cdot \frac{x}{4}=\frac{\ep_1x}{4\,\log^2\big(\frac{15x}{\ep_2d}\big)}\geq |W|.$$ Then by Claim~\ref{claim-ball-expands} applied to $A$ and $B$, we have that $|B^{m_0}_{G-W}(A)|,|B^{m_0}_{G-W}(B)|>n/2$, which implies that there exists an $A,B$-path in $G-W$ with length at most $$2m_0\leq\frac{100}{\eps_1}\log^3\frac{15n}{\ep_2d},$$ as desired. 
    
    If $x < \ep_2d/2$, take vertices $a \in V(A)$ and $b \in V(B)$ and let $N_{G-W}(a) =: A'$ and $N_{G-W}(b) =: B'$. Observe that since $$|W| \leq \frac{\ep_1\ep_2d}{4\log^2(\frac{15n}{\ep_2d})} \leq \ep_2d/4$$ and $d(a), d(b) \geq d-1$, we have that $|A'|, |B'| \geq \ep_2d/2 > x$. Hence we have $\ep(|A'|)\cdot |A'|/4\geq |W|$ and $\ep(|B'|)~\cdot~|B'|/4~\geq~|W|$. Then by Claim~\ref{claim-ball-expands} applied to $A'$ and $B'$, we have that $|B^{m_0}_{G-W}(A')|> n/2$ and $|B^{m_0}_{G-W}(B')|>n/2$, which implies that there exists an $A,B$-path in $G-W$ with length at most $$2m_0+2\leq\frac{100}{\eps_1}\log^3\frac{15n}{\ep_2d},$$ as desired.
    \end{proof}

\subsection{Reducing Theorem~\ref{thm-main}}\label{sec-def-reduction}

Using Lemma~\ref{cor-expander}, we can assume that $G$ is a bipartite expander graph. If $G$ is very dense, a classic result of Alon, Krivelevich and Sudakov~\cite[Theorem 6.1]{alon} provides a balanced $1$-subdivision of a clique on $\Omega(\sqrt{|G|})$ vertices.
\begin{theorem}\label{thm-aks}
Let $\alpha > 0$. 
If $G$ is a graph with $n$ vertices and average degree $\alpha n$, then $G$ contains a copy of $\tk_{r}^{(1)}$ where $r:=\alpha n^{1/2}/2$.
\end{theorem}

For sparse expanders, the following result of Wang~\cite[Lemma $1.3$]{yw} provides a balanced clique subdivision of size linear in its average degree.

\begin{theorem}\label{thm-yanwang}
There exists $\ep_1 > 0$ such that for any $0 < \ep_2 < 1/5$ and $s \geq 20$, there exist $d_0 = d_0(\ep_1, \ep_2, s)$ and some constant $t > 0$ 
such that the following holds for each $n \geq d \geq d_0$ and $d < \log^s n$.
Suppose that $G$ is an $n$-vertex bipartite $(\ep_1, \ep_2 d)$-expander graph with $\delta(G) \geq d$. 
Then $G$ contains a copy of $\tk_{td}^{(\ell)}$ for some $\ell \in \mathbb{N}$.
\end{theorem}

By Theorems~\ref{thm-aks} and~\ref{thm-yanwang}, to prove Theorem~\ref{thm-main} it suffices to prove the following:

\begin{theorem}\label{main-theorem1}
For each $0<\ep_1,\ep_2<1$, the following holds for all sufficiently large~$K$.
Let $G$ be an $n$-vertex bipartite $(\ep_1,\ep_2d)$-expander with $\delta(G)\ge d$, $n\ge Kd$ and $d\ge  \log^{800}n$. Then $G$ contains a copy of $\tk_{\sqrt{d}}^{(\ell)}$ for some $\ell \in \mathbb{N}$.
\end{theorem}

For brevity, throughout this paper we set $$m:=\left\lfloor\log^4\frac{n}{d}\right\rfloor.$$ Actually, we shall prove a stronger version of Theorem~\ref{main-theorem1}:

\begin{theorem}\label{main-theorem}
For each $0<\ep_1,\ep_2<1$,  the following holds for all sufficiently large~$K$.
Let $G$ be an $n$-vertex bipartite $(\ep_1,\ep_2d)$-expander with $\delta(G)\ge d$, $n\ge Kd$ and $d\ge  \log^{800}n$. %\footnote{JH: I think it should really be $\log^{800}(n)$ here rather than $m^{200}$. Indeed, $d < log^{800}(n)$ is handled by Theorem~\ref{thm-yanwang}, and clearly $\log^{800}(n) \geq m^{200}$.}
Then $G$ contains a copy of $\tk_{\sqrt{d}m}^{(\ell)}$ for $\ell=1$ or $\ell=80m^3$.

\end{theorem}

Note that $m\ge \lfloor\log^4 K\rfloor$, so we can choose $K$ sufficiently large to ensure that $m$ is large enough for all later statements and proofs. % Moreover, choosing $K$ such that $\frac{K}{\log^{800}K}\ge 1000$.
% Notice that $f(x)=\frac{x}{\log^{800}x}$ is increasing when $f(x)>1000$, so it can conclude that $n\ge 1000dm^{200}$. \footnote{Not accomplished.}
Moreover, since $n/d \geq K$, taking $K$ sufficiently large we can assume that~$n \geq dm^{200}$.
%In particular, for all $\ell \leq m^5$ we have $|\tk^{(\ell)}_{\sqrt{d}m}| < n$, and so $G$ can feasibly contain a copy of $\tk^{(\ell)}_{\sqrt{d}m}$.

\section{Proof sketch of Theorem~\ref{main-theorem}}\label{sec-proof-ideas}

Assume that the graph $G$ contains no $\tk_{\sqrt{d}m}^{(1)}$. Thus we have to find
a copy of $\tk_{\sqrt{d}m}^{(\ell)}$, where $\ell:=80m^3$. 
One immediate obstruction to  a naive greedy construction is that the desired subdivision  has $\binom{\sqrt{d}m}{2}\ell \geq dm$ vertices which is much larger than our lower bound $d$ on the minimal degree $\delta(G)$. That is, if we were to construct a copy of $\tk_{\sqrt{d}m}^{(\ell)}$ one path at a time, we may arrive at a point when all unused vertices in $G$ have neighbours only internal to the previously constructed paths. This would be overcome if there existed sufficiently many vertices of degree $\Omega\left(\binom{\sqrt{d}m}{2}\ell\right)$, but we cannot guarantee this. 

However, using the expansion property of $G$ (in particular, Lemma~\ref{lem-mpath}) 
%and that $d \geq m^{200}$ 
we can find $\sqrt{d}m$ rooted trees, called \emph{units}, each containing $\Theta(dm^{28})$ leaves, all at the same distance from the root 
%and $O(\sqrt{d}m^8)$ non-leaves, called \emph{units} 
(see Definition~\ref{def-unit} and Figure~\ref{fig-unit}). Also, each constructed unit will have very small \emph{interior} (that is, the set of its non-leaf vertices), namely of size at most $O(\sqrt{d}m^8)$, and any two distinct units  will have disjoint interiors.
Thus we would like to find, for every pair of units,  a path between their \emph{boundaries} (that is, their sets of leaves) so that the paths extended to the roots all have length $\ell$ and are internally disjoint.

For this step, we create
%$\binom{\sqrt{d}m}{2}$ 
certain structures called \emph{adjusters}, introduced by Liu and Montgomery in \cite{lmerdoshajnal}. While the adjusters were constructed in sparse expanders in \cite{lmerdoshajnal}, the bulk of the work in our paper is to construct adjusters in dense expanders whose average degree could be a polynomial of the order of the expander.
Roughly speaking, a \emph{$(j, k)$-adjuster} consists of two units rooted at $x$ and $y$ together with a collection of (not necessarily internally vertex disjoint) $x,y$-paths that have lengths $\ell', \ell'+2, \ell'+4, \ldots, \ell' + 2k$ for some $\ell' \leq j$. We call $\ell'+2k$ the \emph{length} of the adjuster. We first construct an $(O(m), 1)$-adjuster (Lemma~\ref{lem-router-dense-max}). We then chain together $(O(m), 1)$-adjusters to form a $(O(m^3), \Omega(m^2))$-adjuster whose length is contained in some fixed interval of length $o(m^2)$ (Lemma~\ref{lem-chain-dense-max}). Such an $(O(m^3), \Omega(m^2))$-adjuster can then be used along with Lemma~\ref{lem-mpath} to construct a path between the boundaries of any two given units that has the desired length and also avoids any given relatively small set $W$ (Lemma~\ref{lem-precise-path}). Namely, we connect the two unit boundaries to the two opposite ends of the adjuster via short paths and then choose a path of the appropriate length inside the adjuster.

The proof is completed by connecting each pair of roots of the $\sqrt{d}m$ constructed units, one by one in some order, by a path of length $\ell$ through a new $(O(m^3), \Omega(m^2))$-adjuster for each pair as above (satisfying appropriate disjointness conditions in each of these steps). Of course, there are a number of technical issues to take care of such as, for example, making sure that a bulk  of each unit remains available throughout the whole procedure.

\section{Building units}\label{sec-build-units}

\begin{defn}\label{def-unit}
	An \emph{$(h_0,h_1,\ell)$-unit} $F$ consists of a \emph{core} vertex $v$ and $h_0$ pairwise vertex disjoint $h_1$-stars $S_{u_i}$ in $F-v$, with centres $u_i$ respectively, $i\in [h_0]$, along with $v,u_i$-paths $P_i$, which are internally vertex disjoint from each other and $\cup_{i \in [h_0]}V(S_{u_i})$. 
	Furthermore, all paths $P_i$ are of length exactly $s$, for some $s< \ell$.
	We call the union of all vertices in the paths $\inte(F):=\cup_{i\in[h_0]}V(P_i)$ the \emph{interior of $F$}, and $\bd(F):=V(F)\setminus \inte(F)$ the \emph{boundary of $F$}.
    % For each $i\in[h_0]$, $P_i\cup S_{u_i}$ is a \emph{branch} of the unit. 
	We say that two units are \emph{disjoint} if their interiors are vertex-disjoint. 
\end{defn}

\begin{figure}[!ht]
    \centering
    
\begin{tikzpicture}[scale=1.5]
    \node[inner sep= 1pt](v1) at (0,3)[circle,fill]{};
	\node[inner sep= 1pt](v1l) at (-0.3,3)[]{\small $v$};
	\node[inner sep= 1pt](v11) at (2,4)[circle,fill]{};
	\node[inner sep= 1pt](v12) at (2,3)[circle,fill]{};
	\node[inner sep= 1pt](v13) at (2,2)[circle,fill]{};
	\draw[decorate, decoration=snake, segment length=6mm,darkslateblue] (v1) -- (v11);
	\draw[decorate, decoration=snake, segment length=6mm,darkslateblue] (v1) -- (v12);
	\draw[decorate, decoration=snake, segment length=6mm,darkslateblue] (v1) -- (v13);
	\node[inner sep= 1pt](a1) at (2.8,4.3)[circle,fill]{};
	\node[inner sep= 1pt](a2) at (2.8,4.1)[circle,fill]{};
	\node[inner sep= 1pt](a3) at (2.8,3.9)[circle,fill]{};
	\node[inner sep= 1pt](a4) at (2.8,3.7)[circle,fill]{};
	\draw (v11) -- (a1);
	\draw (v11) -- (a2);
	\draw (v11) -- (a3);
	\draw (v11) -- (a4);
	\node[inner sep= 1pt](b1) at (2.8,3.3)[circle,fill]{};
	\node[inner sep= 1pt](b2) at (2.8,3.1)[circle,fill]{};
	\node[inner sep= 1pt](b3) at (2.8,2.9)[circle,fill]{};
	\node[inner sep= 1pt](b4) at (2.8,2.7)[circle,fill]{};
	\draw (v12) -- (b1);
	\draw (v12) -- (b2);
	\draw (v12) -- (b3);
	\draw (v12) -- (b4);
	\node[inner sep= 1pt](c1) at (2.8,2.3)[circle,fill]{};
	\node[inner sep= 1pt,black](c2) at (2.8,2.1)[circle,fill]{};
	%\node[inner sep= 1pt,deepcarmine](c2l) at (3.05,2.05)[]{\small $w_2$};
	\node[inner sep= 1pt](c3) at (2.8,1.9)[circle,fill]{};
	\node[inner sep= 1pt](c4) at (2.8,1.7)[circle,fill]{};
	\draw (v13) -- (c1);
	\draw (v13) -- (c2);
	\draw (v13) -- (c3);
	\draw (v13) -- (c4);
	\draw[rotate=90][dashed,deepcarmine] (3,-2.8) ellipse (50pt and 12pt);
	\node[inner sep= 1pt,deepcarmine](bdl) at (2.8,5){\small $\bd(F)$};
	\node[inner sep= 1pt,cadmiumgreen](intl) at (1.1,5){\small $\inte(F)$};
	\draw[rounded corners=15pt][dashed,cadmiumgreen]
  (-0.15,1.5) rectangle ++(2.4,3);
  \node[inner sep= 1pt,darkslateblue](l) at (0.8,3.8){\small $s$};
  \node[inner sep= 1pt,darkslateblue](l) at (1.1,3.15){\small $s$};
  \node[inner sep= 1pt,darkslateblue](l) at (1.1,2.6){\small $s$};
  \draw[<->,darkslateblue] (0,3.2) -- (1.85,4.1);
\end{tikzpicture}
\caption{A $(3,4,s+1)$-unit $F$.}\label{fig-unit}
\end{figure}
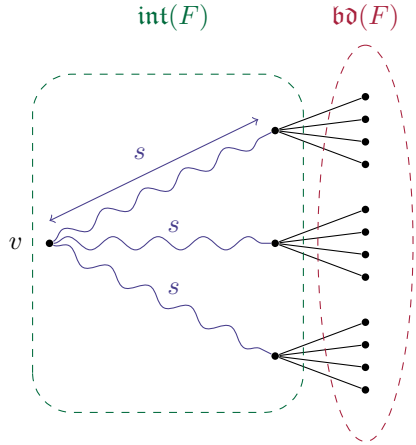

%\textcolor{red}{We can delete this red part if you are happy with my footnote and the change I did in def 2.2.}

%\textcolor{red}{We will need the following robust expansion result of Liu.\footnote{JH: Hong mentioned this (minus the preamble) in one of our meetings recently. Irene (or Hong), do you know where its from/is it a corollary of some previous result?}\footnote{IG: I think this is a definition. Komlos and Szemeredi introduced the notion of expander and Kim-Liu-Haslegrave introduced the notion of \emph{robust expander}, which is the same as our def 2.2 but with the same rate of expansion after deleting a small set of edges (see Def 3.1 in~\cite{HKL20}). Then, what they show is that we can find a a robust-expander subgraph in any graph (see Lemma 3.2 in~\cite{HKL20}) I will change the definition of expander to the def of robust expander, as it generalises the former (just put $F=\varnothing$).}}
%\textcolor{red}{
%\begin{lemma}\label{lem-robust-expansion}
 %    For each $0 < \ep_1, \ep_2 < 1$, there exists $K > 0$ such that the following holds. Let $G$ be an $n$-vertex bipartite $(\ep_1, \ep_2d)$-expander with $\delta(G) \geq d$ and $n \geq Kd$. Then given any $X \subseteq V(G)$ with $\frac{d}{1000} \leq |X| \leq n/2$ and any $F \subseteq E(G)$ with $|F| \leq \frac{d|X|}{m}$ we have that $$|N_{G\setminus F}(X)| \geq \frac{|X|\ep(|X|)}{2}.$$\footnote{JH:Probably better to put a definition of $G \setminus F$ in the notation section.}\footnote{ZW:maybe we need move this lemma to Section 2?}
%\end{lemma}}

See Figure~\ref{fig-unit} for an illustration of a unit. We now show the existence of a unit in a dense graph after removing a relatively small set of vertices.

\begin{lemma}\label{lem-unit}
For each $0<\ep_1,\ep_2<1$, the following holds for all sufficiently large~$K$. 
Let $G$ be an $n$-vertex bipartite $(\ep_1,\ep_2d)$-expander with $\delta(G)\ge d$, $n\ge Kd$ and $d\ge  m^{200}$.
Then, given any $W\subseteq V(G)$ with $|W|\leq dm^{30}$, $G-W$ contains a $(4\sqrt{d}m^6,\sqrt{d}m^{22},10m)$-unit.
\end{lemma}

\begin{proof}
    We will first construct $m^{40} + \sqrt{d}m^{49}$ disjoint stars, $m^{40}$ of which will have $d/m^5$ leaves each and $\sqrt{d}m^{49}$ of which will have $\sqrt{d}m^{24}$ leaves each. 
    We will then show that some collection of stars can be connected together (possibly losing some leaves of the stars in the process) in order to create a $(4\sqrt{d}m^6,\sqrt{d}m^{22},10m)$-unit.
    
\begin{claim}\label{claim-many-stars}
    Let $W \subseteq V(G)$. If $|W|  \leq dm^{74}$, then there exists a vertex $v \in G - W$ such that $d_{G-W}(v) \geq d/m^5$. 
\end{claim}
\begin{poc}
    Suppose not. Then $\Delta(G-W) < d/m^5$. Recall that, since $n \geq Kd$, we must have $n \geq dm^{200}$. Thus, since $|W| \leq dm^{74}$, we can take a subset of vertices $U \subseteq V(G)\setminus W$ with $|U| = dm^{76}$. Define $F$ be the graph on $V(G)\setminus W$ with edge set $$E(F) := \{uv\in E(G):u \in U, v \in V(G)\setminus W\}$$ and no isolated vertices. 
    Then 
    $$e(F) \leq \Delta(G-W)\cdot |U| \leq \frac{d\,|U|}{m^5}  \leq d(G)\cdot \ep(|U|,\ep_1,\ep_2d)\cdot |U|.$$
    Observe that $N_{G\setminus F}(U) \subseteq W$. Hence $|N_{G\setminus F}(U)| \leq |W| \leq dm^{74}$. But, as $G$ is an expander, we have that $|N_{G\setminus F}(U)| \geq \frac{1}{2}|U|\,\ep(|U|,\ep_1,\ep_2d) \geq  dm^{75}$, which is a contradiction.
\end{poc}

Now, iteratively apply Claim~\ref{claim-many-stars} to $G-W$ a total of $m^{40}$ times, at each iteration $i$, $1 \leq i \leq m^{40}$, removing a star $S_i$ with centre $u_i$ and $d/m^5$ leaves from the current graph and adding it to a set~$S^1$. 
Let $V(S^1)$ be the set of vertices in the stars in $S^1$ and observe that all stars in $S^1$ are vertex disjoint. 
Observe that one can do this because throughout the process $|V(S^1)| \leq dm^{35}$ and thus $|W \cup V(S^1)| \leq 2dm^{40} \leq dm^{73}$ also.

Next, we iteratively apply Claim~\ref{claim-many-stars} to $G-(W\cup S^1)$ a further $\sqrt{d}m^{49}$ times, at each iteration $i$, $m^{40}+1\leq i \leq m^{40}+\sqrt{d}m^{49}$, removing a star $S_i$ with centre $u_i$ and $\sqrt{d}m^{24}$ leaves from the current graph and adding it to a set $S^2$. 
Let $V(S^2)$ be the set of vertices in the stars in $S^2$ and observe that all stars in $S^2$ are vertex disjoint. 
Observe that one can do this since $d \geq m^{200}$ implies $\sqrt{d}m^{24} \leq d/m^5$, and throughout the process $|V(S^2)| \leq dm^{73}$ and thus $|W \cup V(S^1) \cup V(S^2)| \leq dm^{74}$ also. 

So now we have in $G-W$ a set $S^1$ of $m^{40}$ stars each with $d/m^5$ leaves and a set $S^2$ of $\sqrt{d}m^{49}$ stars each with $\sqrt{d}m^{24}$ leaves, such that all stars in $S^1 \cup S^2$ are vertex disjoint from each other. 
Let $\ell:=m^{40}+\sqrt{d}m^{49}$. 
Take $I\subseteq [m^{40}] \times [\ell]$ to be a maximal subset for which there are paths $P_{(i,j)}$, $(i,j)\in I$, in $G-W$ such that the following holds.

\begin{itemize}
\item For each $(i,j)\in I$, $P_{(i,j)}$ is a $u_i,u_j$-path with length at most $2m$ which is disjoint from $\{u_k:k\in [\ell]\setminus\{i,j\}\}$.
\item The paths are internally vertex disjoint.
%For each $\{i,j\},\{i,j'\}\in I$ with $j\neq j$, $V(P_{\{i,j\}})\cap V(P_{\{i,j\}})=\{i\}$.
\end{itemize}

Suppose there is some $i\in [m^{40}]$ and  $J\subseteq [\ell]$ with $(i,j)\in I$, for each $j\in J$, with $|J|= \sqrt{d}m^{8}$.
Let $U:=\cup_{j\in J}V(P_{(i,j)})$, and note that $|U|\leq 2\cdot \sqrt{d}m^{9}$. 
Hence, for any $j\in J$, $$|V(S_j)\setminus U|\geq \sqrt{d}m^{24}/2 \geq \sqrt{d}m^{23}.$$ %(since $n \geq Kd$). 
By pigeonhole, we can pick a subset $J'\subseteq J$ such that all paths $P_{(i,j)}$ with $j\in J'$ have the same length, and $|J'|\geq |J|/2m \geq 4\sqrt{d}m^6$.
Taking stars $S_{j'}\subseteq S_j-U$, $j\in J'$, with $\sqrt{d}m^{22}$ leaves, and the paths $P_{(i,j)}$, $j\in J'$, we get a $(4\sqrt{d}m^6,\sqrt{d}m^{22},10m)$-unit.

Suppose then that there are no such $i\in [m^{40}]$ and $J\subseteq [\ell]$.
Let $J_1 := \{u_1, \ldots, u_{m^{40}}\}$, and let $J_2\subseteq [\ell]$ be a maximal set such that there is no $j_1\in J_1$ and $j_2\in J_2$ with $(j_1,j_2)\in I$. 
Then, $|J_2|\ge \ell-|J_1|\cdot\sqrt{d}m^{8}\ge \sqrt{d}m^{18}$. 
Let 
$$
U:=\bigcup_{i\in J_1,  \{i,j\}\in I }V(P_{(i,j)}\setminus \{u_i,u_j\}).
$$
Using that $d \geq m^{200}$ we have that $|U|\le m^{40}\cdot \sqrt{d}m^{8}\cdot 2m\leq dm^{30}$, and hence $|W\cup U|\leq 2dm^{30}$.
Also, observe that $$|\cup_{i\in J_1}(V(S_i)\setminus U)|\ge m^{40}\cdot d/m^5-dm^{30}\ge dm^{34}\ge \log^3\left(\frac{n}{d}\right)|W\cup U|$$ and

$$|\cup_{i\in J_2}(V(S_i)\setminus U)|\ge \sqrt{d}m^{49}\cdot \sqrt{d}m^{24}-dm^{30}\ge dm^{71}\ge \log^3\left(\frac{n}{d}\right)|W\cup U|.$$
Hence, by Lemma~\ref{lem-mpath}, and that $n \geq Kd$ with $K$ sufficiently large, we can find a path connecting $\cup_{i\in J_1}(V(S_i)\setminus U)$ and $\cup_{i\in J_2}(V(S_i)\setminus U)$ which avoids $W\cup U$ and has length at most $2m$.
This contradicts the maximality of $I$.
\end{proof}

\section {Building adjusters}\label{sec-build-adj}

In order to ensure the clique subdivision we construct is balanced, we utilise special structures called \emph{adjusters}, introduced by Liu and Montgomery in \cite{lmerdoshajnal}.% These adjusters have many paths with different (consecutive) lengths between two vertices. %, allowing one to choose a path of desired length

\begin{defn}
    \sloppy{An $(\ell,k)$-adjuster $P = (v_1, v_2, F_1, F_2, z, A,\mathcal P)$ in a graph $G$ consists of two vertices $v_1, v_2$ that are the core vertices of two $(z\sqrt{d}m^6,\sqrt{d}m^{22},10m)$-units $F_1$ and $F_2$, respectively, where $|\inte(F_1) \cap \inte(F_2)| =\varnothing$, a real number $z \in [1,3]$, a vertex set $A \subseteq V(G)$ of size at most~$2\ell$
    disjoint from $\cup_{i\in[2]}V(F_i)$, and a
    collection $\mathcal P$ of $k+1$ $v_1,v_2$-paths in $G[A \cup \{v_1, v_2\}]$ of lengths $\ell', \ell'+2, \ell'+4, \ldots, \ell' + 2k$ for some $\ell' \leq \ell$.}
    Further, define $\ell(P):=\ell'+2k$ and call it the \emph{length of $P$}. We also define the \emph{perimeter} $p(P) := A\cup \inte(F_1)\cup \inte(F_2)$.
\end{defn}

Note that the condition $\ell'\le \ell$ from the definition of an $(\ell,k)$-adjuster $P$ is equivalent to $\ell(P)\le \ell+2k$ (and this will be the form in which we will be verifying it).
We say that an adjuster $P$ is `in a set $X$' to mean that $p(P)$ is a subset of $X$. Also, we may just write $P=(v_1, v_2, F_1, F_2,z, A)$ when the collection $\mathcal P$ of paths is understood. We may additionally omit $z$ if the size of the units is unimportant.\footnote{Such as at the end of the proof of Lemma~\ref{lem-precise-path}.}

\definecolor{aurometalsaurus}{rgb}{0.43, 0.5, 0.5}
\definecolor{blue(munsell)}{rgb}{0.0, 0.5, 0.69}
\definecolor{emerald}{rgb}{0.31, 0.78, 0.47}
\begin{figure}[!ht]
    \centering
\begin{tikzpicture}
    \draw[aurometalsaurus] (0,0) circle (0.5cm);
	\draw[aurometalsaurus] (2,0) circle (0.5cm);
	\draw[aurometalsaurus] (4,0) circle (0.5cm);
	\draw[aurometalsaurus] (6,0) circle (0.5cm);
	\draw[decorate, decoration=snake, segment length=4mm,aurometalsaurus] (-1.5,0)--(-0.5,0);
	\draw[decorate, decoration=snake, segment length=4mm,aurometalsaurus] (0.5,0)--(1.5,0);
	\draw[decorate, decoration=snake, segment length=4mm,aurometalsaurus] (2.5,0)--(3.5,0);
	\draw[decorate, decoration=snake, segment length=4mm,aurometalsaurus] (4.5,0)--(5.5,0);
	%%longest-path
	\draw[decorate, decoration=snake, segment length=4mm,red,densely dotted,thick] (-1.5,0)--(-0.5,0);
	\draw[decorate, decoration=snake, segment length=4mm,red,densely dotted,thick] (0.5,0)--(1.5,0);
	\draw[decorate, decoration=snake, segment length=4mm,red,densely dotted,thick] (2.5,0)--(3.5,0);
	\draw[decorate, decoration=snake, segment length=4mm,red,densely dotted,thick] (4.5,0)--(5.5,0);
	\begin{scope}
    \clip (-0.5,0) rectangle (0.5,0.5);
    \draw[red,densely dotted,thick] (0,0) circle(0.5cm);
    \end{scope}
    \begin{scope}
    \clip (2.5,0) rectangle (-1.5,0.5);
    \draw[red,densely dotted,thick] (2,0) circle(0.5cm);
    \end{scope}
    \begin{scope}
    \clip (3.5,0) rectangle (4.5,0.5);
    \draw[red,densely dotted,thick] (4,0) circle(0.5cm);
    \end{scope}
    \begin{scope}
    \clip (5.5,0) rectangle (6.5,0.5);
    \draw[red,densely dotted,thick] (6,0) circle(0.5cm);
    \end{scope}
    %%right-unit
    \node[inner sep= 1pt,blue(munsell)](v1) at (6.5,0)[circle,fill]{};
	\node[inner sep= 1pt,blue(munsell)](v1l) at (6.2,0){\small $v_2$};
	\node[inner sep= 1pt,blue(munsell)](v11) at (8.5,1)[circle,fill]{};
	\node[inner sep= 1pt,blue(munsell)](v12) at (8.5,0)[circle,fill]{};
	\node[inner sep= 1pt,blue(munsell)](v13) at (8.5,-1)[circle,fill]{};
	\draw[decorate, decoration=snake, segment length=6mm,blue(munsell)] (v1) -- (v11);
	\draw[decorate, decoration=snake, segment length=6mm,blue(munsell)] (v1) -- (v12);
	\draw[decorate, decoration=snake, segment length=6mm,blue(munsell)] (v1) -- (v13);
	\node[inner sep= 1pt,blue(munsell)](a1) at (9.3,1.3)[circle,fill]{};
	\node[inner sep= 1pt,blue(munsell)](a2) at (9.3,1.1)[circle,fill]{};
	\node[inner sep= 1pt,blue(munsell)](a3) at (9.3,0.9)[circle,fill]{};
	\node[inner sep= 1pt,blue(munsell)](a4) at (9.3,0.7)[circle,fill]{};
	\draw[blue(munsell)] (v11) -- (a1);
	\draw[blue(munsell)] (v11) -- (a2);
	\draw[blue(munsell)] (v11) -- (a3);
	\draw[blue(munsell)] (v11) -- (a4);
	\node[inner sep= 1pt,blue(munsell)](b1) at (9.3,0.3)[circle,fill]{};
	\node[inner sep= 1pt,blue(munsell)](b2) at (9.3,0.1)[circle,fill]{};
	\node[inner sep= 1pt,blue(munsell)](b3) at (9.3,-0.1)[circle,fill]{};
	\node[inner sep= 1pt,blue(munsell)](b4) at (9.3,-0.3)[circle,fill]{};
	\draw[blue(munsell)] (v12) -- (b1);
	\draw[blue(munsell)] (v12) -- (b2);
	\draw[blue(munsell)] (v12) -- (b3);
	\draw[blue(munsell)] (v12) -- (b4);
	\node[inner sep= 1pt,blue(munsell)](c1) at (9.3,-0.7)[circle,fill]{};
	\node[inner sep= 1pt,blue(munsell)](c2) at (9.3,-0.9)[circle,fill]{};
	\node[inner sep= 1pt,blue(munsell)](c3) at (9.3,-1.1)[circle,fill]{};
	\node[inner sep= 1pt,blue(munsell)](c4) at (9.3,-1.3)[circle,fill]{};
	\draw[blue(munsell)] (v13) -- (c1);
	\draw[blue(munsell)] (v13) -- (c2);
	\draw[blue(munsell)] (v13) -- (c3);
	\draw[blue(munsell)] (v13) -- (c4);
    %%left-unit
    \node[inner sep= 1pt,blue(munsell)](v2) at (-1.5,0)[circle,fill]{};
	\node[inner sep= 1pt,blue(munsell)](v2l) at (-1.45,0.3)[]{\small $v_1$};
	
	\node[inner sep= 1pt,blue(munsell)](v21) at (-3.5,1)[circle,fill]{};
	\node[inner sep= 1pt,blue(munsell)](v22) at (-3.5,0)[circle,fill]{};
	\node[inner sep= 1pt,blue(munsell)](v23) at (-3.5,-1)[circle,fill]{};
	\draw[decorate, decoration=snake, segment length=6mm,blue(munsell)] (v2) -- (v21);
	\draw[decorate, decoration=snake, segment length=6mm,blue(munsell)] (v2) -- (v22);
	\draw[decorate, decoration=snake, segment length=6mm,blue(munsell)] (v2) -- (v23);
	\node[inner sep= 1pt,blue(munsell)](d1) at (-4.3,1.3)[circle,fill]{};
	\node[inner sep= 1pt,blue(munsell)](d2) at (-4.3,1.1)[circle,fill]{};
	\node[inner sep= 1pt,blue(munsell)](d3) at (-4.3,0.9)[circle,fill]{};
	\node[inner sep= 1pt,blue(munsell)](d4) at (-4.3,0.7)[circle,fill]{};
	\draw[blue(munsell)] (v21) -- (d1);
	\draw[blue(munsell)] (v21) -- (d2);
	\draw[blue(munsell)] (v21) -- (d3);
	\draw[blue(munsell)] (v21) -- (d4);
	\node[inner sep= 1pt,blue(munsell)](e1) at (-4.3,0.3)[circle,fill]{};
	\node[inner sep= 1pt,blue(munsell)](e2) at (-4.3,0.1)[circle,fill]{};
	\node[inner sep= 1pt,blue(munsell)](e3) at (-4.3,-0.1)[circle,fill]{};
	\node[inner sep= 1pt,blue(munsell)](e4) at (-4.3,-0.3)[circle,fill]{};
	\draw[blue(munsell)] (v22) -- (e1);
	\draw[blue(munsell)] (v22) -- (e2);
	\draw[blue(munsell)] (v22) -- (e3);
	\draw[blue(munsell)] (v22) -- (e4);
	\node[inner sep= 1pt,blue(munsell)](f1) at (-4.3,-0.7)[circle,fill]{};
	\node[inner sep= 1pt,blue(munsell)](f2) at (-4.3,-0.9)[circle,fill]{};
	\node[inner sep= 1pt,blue(munsell)](f3) at (-4.3,-1.1)[circle,fill]{};
	\node[inner sep= 1pt,blue(munsell)](f4) at (-4.3,-1.3)[circle,fill]{};
	\draw[blue(munsell)] (v23) -- (f1);
	\draw[blue(munsell)] (v23) -- (f2);
	\draw[blue(munsell)] (v23) -- (f3);
	\draw[blue(munsell)] (v23) -- (f4);
	%%labelling
	\draw[rounded corners=15pt][dashed,emerald]
  (-3.8,-1.2) rectangle ++(12.7,2.4);
  \node[inner sep= 1pt,emerald](pl) at (3,-0.9){\small $p(P)$};
  \node[inner sep= 1pt,red](pl) at (3,0.7){\small $\ell(P)$};
  \node[inner sep= 1pt,blue(munsell)](f1l) at (-2.1,0.7){\small $F_1$};
  \node[inner sep= 1pt,blue(munsell)](f2l) at (7.1,0.7){\small $F_2$};
\end{tikzpicture}
\caption{An example of an adjuster $P$ we shall build, in which $A$ is a string of even cycles, joined by paths connecting at almost antipodal vertices of the cycles. The length $\ell(P)$ is found by traversing all the longest paths between the almost antipodal vertices on the cycles.}\label{fig-lkadjuster}
\end{figure}
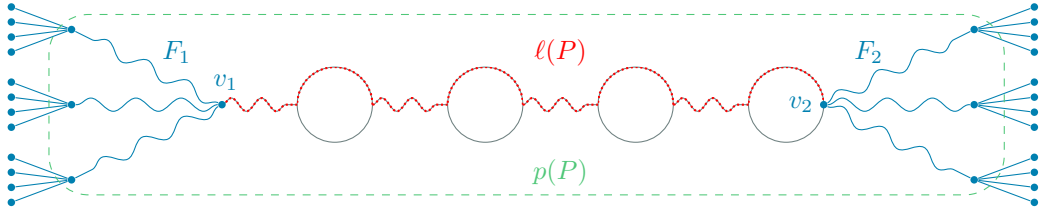

\subsection{Asymmetric bipartite parts}

To construct robustly many adjusters in our expander, we need the following result showing that a dense asymmetric bipartite subgraph is enough for finding a 1-subdivision of large clique.

\begin{prop}\label{littleprop}
Let $d\ge 40$ and suppose that a graph $G$ contains disjoint vertex sets $U$ and $W$ such that every vertex in $U$ has at least $d$ neighbours in $W$. 
Let $\ell:=\lceil d\sqrt{|U|}/8|W|\rceil$ and suppose $\ell\geq 20$. Then, $G$ contains a copy of $\tk_{\ell}^{(1)}$.
\end{prop}

\begin{proof}
If $|U|>16|W|^2$, then replace $U$ by a subset of size exactly $16|W|^2$; the new $\ell=\lceil d/2\rceil$ is still at least $20$.

Let $p:= \sqrt{|U|}/4|W|\le 1$, so that $\lceil pd/2\rceil=\ell$, and hence $pd/12\geq 2$. 
Let $W'\subseteq W$ be a subset formed by choosing each element of $W$ independently at random with probability $p$.
Set $X_u:=d_G(u,W')=\sum_{w\in N(u,W)} \mathds{1}_{\{w\in W'\}},$ which is a sum of independent Bernoulli random variables of parameter $p$ with expectation $\bE[X_u]\geq pd$.
Then, using a standard lower-tail Chernoff bound (e.g.\ \cite[Theorem 4.5]{MU05}) gives that, for each $u\in U$,
\[
\P(X_u< pd/2)\leq \exp(-pd/12)\leq 1/4,
\] using that $pd/12 \geq 2$.
Therefore, letting $U':=\{u\in U:X_u\geq pd/2\}$, it is clear that we have $\bE\left[\,|U'|\,\right]=|U|\cdot\P(\,X_u\geq pd/2\,)\geq 3|U|/4$. 
The last inequality (combined with $|U'| \leq |U|$)  gives that $$\bP(\,|U'| < |U|/4\,) \leq 1/2.$$
Now, we have $|W|\geq d$, and hence $\bE\left[\,|W'|\,\right]=p|W|\geq 24$. Thus, using an upper-tail Chernoff bound (e.g.\ \cite[Theorem 4.4]{MU05}), we have that 
$$\P(\,|W'|>2p|W|\,)\leq \exp(-p|W|/3)\leq 1/4.$$
And so we have that $\P(\,|W'|>2p|W|\,)+\P(\,|U'|<|U|/4\,)<1$.
Therefore, there is some choice of $W'$ for which $|U'|\geq |U|/4$ and $|W'|\leq 2p|W|$. 

Take a maximal set of pairs $I\subseteq \binom{W'}{2}$ for which there is a set of distinct vertices $v_{\{x,y\}}$ in $U'$, $\{x,y\}\in I$, such that $x,y\in N(v_{\{x,y\}})$ for each $\{x,y\}\in I$.\vspace{0.08cm} 
Noting that $16p^2|W|^2=|U|$, we have $|U'|\geq |U|/4\geq 4p^2|W|^2\geq |W'|^2$.
Thus, there is some $u\in U'\setminus \{v_{\{x,y\}}:\{x,y\}\in I\}$.
Let $A:=N(u,W')$. Then $|A|\geq \lceil pd/2\rceil= \ell$.
Moreover, $\binom{A}{2}\subseteq I$, by the maximality of $I$.
Note that $A\cup \{v_{\{x,y\}}:\{x,y\}\in \binom{A}{2}\}$ is the vertex set of a copy of $\tk_{\ell}^{(1)}$ in $G$ with core vertices those in $A$ and edge set $\{xv_{\{x,y\}},yv_{\{x,y\}}:\{x,y\}\in \binom{A}{2}\}$.
\end{proof}

%%%%%%%%%%%%%%%%%%%%%%%%%%%%%%%%%%%%%%

\subsection{Constructing an adjuster}\label{sec-simple-adj}

\begin{lemma}\label{lem-router-dense-max}
For each $0<\ep_1,\ep_2<1$,  the following holds for all sufficiently large $K$.
For $d > 0$, let $G$ be an $n$-vertex bipartite $(\ep_1,\ep_2d)$-expander with $\delta(G)\ge d$ and $n\ge Kd$. 
Suppose $d\ge m^{200}$ and $G$ contains no copy of  $\tk_{\sqrt{d}m}^{(1)}$.
Let $W\subseteq V(G)$ satisfy $|W|\leq dm^{11}$.
Then, there exists a $(50m,1)$-adjuster $P=(u_1,u_2,F_1,F_2,3,A,\mathcal{P})$ in $G-W$. Moreover, $|A| \leq 2\ell(P)$.
\end{lemma}

\begin{proof}  
By Lemma~\ref{lem-unit}, since $|W|\leq dm^{11}\leq dm^{30}$, we can find a $(3\sqrt{d}m^6+2,\sqrt{d}m^{22},10m)$-unit $F_1$ in $G-W$.
Now, set $
%W_{\ref{lem-unit}}=
W_0:=W\cup V(F_1)$ and use again Lemma~\ref{lem-unit} to find a $(3\sqrt{d}m^6+2,\sqrt{d}m^{22},10m)$-unit $F_2$ in $G-W_0$. 
Repeat the process one more time to find a $(4\sqrt{d}m^6,\sqrt{d}m^{22},10m)$-unit $F_3$ in $G-W_1$, with $W_1:=W_0\cup V(F_2)$. 
This can be done, since $|W_0|\leq|W|+|V(F_1)|\leq dm^{11}+40dm^{29}\leq dm^{30}$ and, similarly, $|W_1|\leq dm^{11}+80dm^{29}\leq dm^{30}$. 

Set $W':=W\cup \inte(F_1)\cup \inte(F_2)\cup \inte(F_3)$ and denote the core vertices of the units $F_1,F_2,F_3$ by $v_1,v_2,v_3$ respectively.
%\footnote{OP: I think these paths should also avoid boundaries of the units. 
%Our later use of the found adjuster $P=(u_1,u_2,...)$ is that we first find some incoming paths, say coming to $P$ at some boundary points $w_1$ and $w_2$. Then we connect $w_i$ to the core vertex $u_i$ by the path in the unit. Finally, we have to connect $u_1$ and $u_2$ by a path $Q$ of the desired length. If $A$ can intersect the boundaries, how do we know that $Q$ does not contain $w_1$ or $w_2$?}

Note that $|W'|\leq 2dm^{11}$. As $|\bd(F_1)|,|\bd(F_2)|\geq dm^{28}$, and recalling that $n \geq Kd$ for sufficiently large $K$ implies $n \geq dm^{200}$, iteratively applying Lemma~\ref{lem-mpath} we can find in $G-W'$ a collection of $dm^{26}$ pairwise vertex disjoint $\bd(F_1),\bd(F_2)$-paths $\mathcal{P}'$, each of length at most $m$. 
By averaging, there exists a subcollection $\cP\subseteq \cP'$ of $dm^{25}$ paths of equal length.
Let $B:=V(\cP)\cap \bd(F_1)$ be the set of endpoints of $\cP$ in $\bd(F_1)$, so that $|B|=dm^{25}$.

Suppose first that there is some vertex $w\notin W'$ with two neighbours $w_1,w_2$ in $B$. Let $P_1,P_2\in \cP$ be the paths that $w_i$ is an endvertex of $P_i$ for each $i\in [2]$ (see Figure~\ref{fig-build-units-v1-v2}).
Denote the paths in $F_1$ joining~$v_1$ with $w_1$ and $w_2$ by $Q_1$ and $Q_2$, respectively.
Denote the paths in $F_2$ joining $P_1\cap\bd(F_2)$ and $P_2\cap\bd(F_2)$ with $v_2$ as $R_1$ and $R_2$, respectively. (Note that these paths may meet before reaching the core vertex.)
By symmetry, we can assume that $w\notin P_2$. 
Then, $v_1Q_1P_1R_1v_2$ and $v_1Q_1ww_2P_2R_2v_2$ are two $v_1,v_2$-paths whose lengths differ by two and are at most 
$$
e(Q_1)+2+e(P_2)+e(R_1)\le 
%(10m-1)+2+m+(10m-1)\le 
50m+2.
$$
Let $F_1'$ be the $(3\sqrt{d}m^6,\sqrt{d}m^{22},10m)$-unit with core vertex $v_1$ constructed from $F_1$ by removing the paths $Q_1$ and $Q_2$ and the leaves of the stars incident to $Q_1$ and $Q_2$. Similarly, let $F_2'$ be the $(3\sqrt{d}m^6,\sqrt{d}m^{22},10m)$-unit with core vertex $v_2$ constructed from $F_2$ by removing the paths $R_1$ and $R_2$ and the leaves of the stars incident to $R_1$ and $R_2$. If $Q_1$ and $Q_2$ (resp.\ $R_1$ and $R_2$) only differ by an edge, then remove in addition an arbitrary path and its leaves from $F_1$ (resp.\ $F_2$) to construct $F_1'$ (resp.\ $F_2'$). Let
%\footnote{OP: I removed $v_1,v_2$ as our def requires $A$ to be disjoint from units.} 
$$
A':=(V(P_1)\cup V(P_2)\cup\{w\}\cup V(Q_1)\cup V(R_1)\cup V(R_2))\setminus\{v_1,v_2\}.
$$
Note that $|A'|\le 2|P_2|+1+30m\le  2\cdot 50m$.
Therefore, 
$
(v_1,v_2,F_1',F_2',3,A')
$ 
is a $(50m,1)$-adjuster in $G-W$. Moreover,
$$
|A'|\le |P_1|+|P_2|+1+|Q_1|-2+|R_1\cup R_2|-3\le 2(|Q_1|+|P_2|+|R_1|-1)= 2\ell(P), $$
so the additional property also holds.
%Also, its length is $20m+|P_2|+3\ge |A'|/2$, as required.

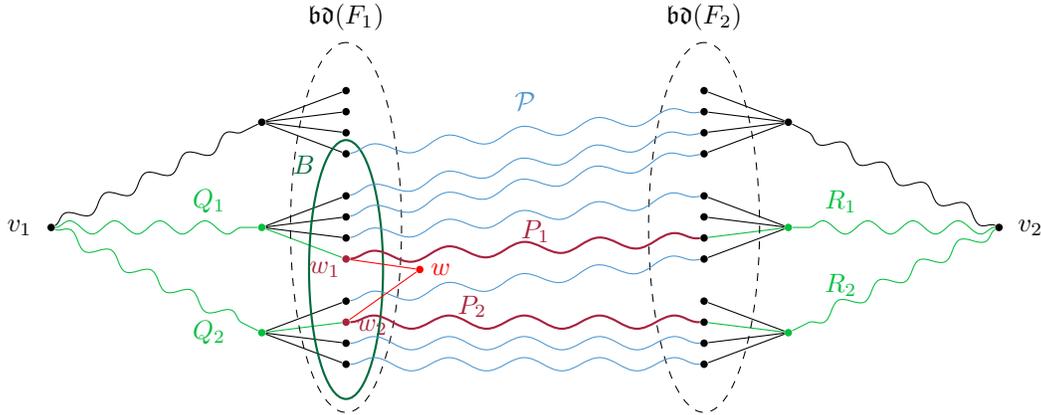
\begin{figure}[!ht]
	\centering
	\begin{tikzpicture}[scale=1.4]
	%unit1
	\node[inner sep= 1pt](v1) at (0,3)[circle,fill]{};
	\node[inner sep= 1pt](v1l) at (-0.3,3)[]{\small $v_1$};
	\node[inner sep= 1pt](v11) at (2,4)[circle,fill]{};
	\node[inner sep= 1pt,darkpastelgreen](v12) at (2,3)[circle,fill]{};
	\node[inner sep= 1pt,darkpastelgreen](v13) at (2,2)[circle,fill]{};
	\node[inner sep= 1pt,darkpastelgreen](q1l) at (1.5,3.25)[]{\small $Q_1$};
	\node[inner sep= 1pt,darkpastelgreen](q2l) at (1.5,2)[]{\small $Q_2$};
	\draw[decorate, decoration=snake, segment length=6mm,black] (v1) -- (v11);
	\draw[decorate, decoration=snake, segment length=6mm,darkpastelgreen] (v1) -- (v12);
	\draw[decorate, decoration=snake, segment length=6mm,darkpastelgreen] (v1) -- (v13);
	\node[inner sep= 1pt](a1) at (2.8,4.3)[circle,fill]{};
	\node[inner sep= 1pt](a2) at (2.8,4.1)[circle,fill]{};
	\node[inner sep= 1pt](a3) at (2.8,3.9)[circle,fill]{};
	\node[inner sep= 1pt](a4) at (2.8,3.7)[circle,fill]{};
	\draw (v11) -- (a1);
	\draw (v11) -- (a2);
	\draw (v11) -- (a3);
	\draw (v11) -- (a4);
	\node[inner sep= 1pt](b1) at (2.8,3.3)[circle,fill]{};
	\node[inner sep= 1pt](b2) at (2.8,3.1)[circle,fill]{};
	\node[inner sep= 1pt](b3) at (2.8,2.9)[circle,fill]{};
	\node[inner sep= 1pt,deepcarmine](b4) at (2.8,2.7)[circle,fill]{};
	\node[inner sep= 1pt,deepcarmine](b4l) at (2.6,2.6)[]{\small $w_1$};
	\draw (v12) -- (b1);
	\draw (v12) -- (b2);
	\draw (v12) -- (b3);
	\draw[darkpastelgreen] (v12) -- (b4);
	\node[inner sep= 1pt](c1) at (2.8,2.3)[circle,fill]{};
	\node[inner sep= 1pt,deepcarmine](c2) at (2.8,2.1)[circle,fill]{};
	\node[inner sep= 1pt,deepcarmine](c2l) at (3.05,2.05)[]{\small $w_2$};
	\node[inner sep= 1pt](c3) at (2.8,1.9)[circle,fill]{};
	\node[inner sep= 1pt](c4) at (2.8,1.7)[circle,fill]{};
	\draw (v13) -- (c1);
	\draw[darkpastelgreen] (v13) -- (c2);
	\draw (v13) -- (c3);
	\draw (v13) -- (c4);
	\draw[rotate=90][dashed] (3,-2.8) ellipse (50pt and 15pt);
	\draw[rotate=90,cadmiumgreen,thick] (2.6,-2.8) ellipse (35pt and 10pt);
	\node[inner sep= 1pt,cadmiumgreen](Al) at (2.4,3.6)[]{\small $B$};
	%unit2	
	\node[inner sep= 1pt](v2) at (9,3)[circle,fill]{};
	\node[inner sep= 1pt](v2l) at (9.3,3)[]{\small $v_2$};
	\node[inner sep= 1pt](v21) at (7,4)[circle,fill]{};
	\node[inner sep= 1pt,darkpastelgreen](v22) at (7,3)[circle,fill]{};
	\node[inner sep= 1pt,darkpastelgreen](v23) at (7,2)[circle,fill]{};
	\node[inner sep= 1pt,darkpastelgreen](r1l) at (7.5,3.25)[]{\small $R_1$};
	\node[inner sep= 1pt,darkpastelgreen](r2l) at (7.5,2.45)[]{\small $R_2$};
	\draw[decorate, decoration=snake, segment length=6mm] (v2) -- (v21);
	\draw[decorate, decoration=snake, segment length=6mm,darkpastelgreen] (v2) -- (v22);
	\draw[decorate, decoration=snake, segment length=6mm,darkpastelgreen] (v2) -- (v23);
	\node[inner sep= 1pt](d1) at (6.2,4.3)[circle,fill]{};
	\node[inner sep= 1pt](d2) at (6.2,4.1)[circle,fill]{};
	\node[inner sep= 1pt](d3) at (6.2,3.9)[circle,fill]{};
	\node[inner sep= 1pt](d4) at (6.2,3.7)[circle,fill]{};
	\draw (v21) -- (d1);
	\draw (v21) -- (d2);
	\draw (v21) -- (d3);
	\draw (v21) -- (d4);
	\node[inner sep= 1pt](e1) at (6.2,3.3)[circle,fill]{};
	\node[inner sep= 1pt](e2) at (6.2,3.1)[circle,fill]{};
	\node[inner sep= 1pt](e3) at (6.2,2.9)[circle,fill]{};
	\node[inner sep= 1pt](e4) at (6.2,2.7)[circle,fill]{};
	\draw (v22) -- (e1);
	\draw (v22) -- (e2);
	\draw[darkpastelgreen] (v22) -- (e3);
	\draw (v22) -- (e4);
	\node[inner sep= 1pt](f1) at (6.2,2.3)[circle,fill]{};
	\node[inner sep= 1pt](f2) at (6.2,2.1)[circle,fill]{};
	\node[inner sep= 1pt](f3) at (6.2,1.9)[circle,fill]{};
	\node[inner sep= 1pt](f4) at (6.2,1.7)[circle,fill]{};
	\draw (v23) -- (f1);
	\draw[darkpastelgreen] (v23) -- (f2);
	\draw (v23) -- (f3);
	\draw (v23) -- (f4);
	\draw[rotate=90][dashed] (3,-6.2) ellipse (50pt and 15pt);
	%paths joining boundaries of units
	\draw[decorate, decoration=snake, segment length=10mm,celestialblue] (a4) -- (d2);
	\draw[decorate, decoration=snake, segment length=10mm,celestialblue] (b1) -- (d3);
	\draw[decorate, decoration=snake, segment length=10mm,celestialblue] (b2) -- (d4);
	\draw[decorate, decoration=snake, segment length=10mm,celestialblue] (b3) -- (e1);
	\draw[decorate, decoration=snake, segment length=10mm,deepcarmine,thick] (b4) -- (e3);
	\draw[decorate, decoration=snake, segment length=10mm,celestialblue] (c1) -- (e4);
	\draw[decorate, decoration=snake, segment length=10mm,deepcarmine,thick] (c2) -- (f2);
	\draw[decorate, decoration=snake, segment length=10mm,celestialblue] (c3) -- (f3);
	\draw[decorate, decoration=snake, segment length=10mm,celestialblue] (c4) -- (f4);
	\node[inner sep= 1pt,celestialblue](P) at (4.5,4.2)[]{\small $\mathcal{P}$};
	\node[inner sep= 1pt,red](w) at (3.5,2.6)[circle,fill]{};
	\node[inner sep= 1pt,red](wl) at (3.7,2.6)[]{\small $w$};
	\draw[red] (w) -- (b4);
	\draw[red] (w) -- (c2);
	\node[inner sep= 1pt,deepcarmine](P1l) at (4.6,2.96)[]{\small $P_1$};
	\node[inner sep= 1pt,deepcarmine](P2l) at (4,2.25)[]{\small $P_2$};
	\node[inner sep= 1pt,black](bd1l) at (2.8,5)[]{\small $\bd(F_1)$};
	\node[inner sep= 1pt,black](bd2l) at (6.2,5)[]{\small $\bd(F_2)$};
	\end{tikzpicture}
	\caption{An illustration of the first case in the proof of Lemma~\ref{lem-router-dense-max}.}\label{fig-build-units-v1-v2}
\end{figure}

Suppose then there is no such vertex.
Let $B_0$ be the set of vertices in $B$ with at least $d/2$ neighbours in $W'$.
As $G$ is $\tk_{\sqrt{d}m}^{(1)}$-free and $\sqrt{d}m \geq 20$, by Proposition \ref{littleprop}, we have % either $(d/2)\sqrt{|B_0|}/8|W'|\leq \sqrt{d}m$, or
\[
\sqrt{d}m\geq \frac{(d/2)\sqrt{|B_0|}}{8|W'|}\geq \frac{\sqrt{|B_0|}}{32m^{11}},
\]
and hence $|B_0|\leq 2^{10}dm^{24}\leq |B|/3$. 
Let $B':=B\setminus B_0$, so that $|B'|\geq 2|B|/3$ and each vertex in $B'$ has at most $d/2$ neighbours in $W'$.

\sloppypar{Now, remove from $B'$ at most $3\sqrt{d}m^6 + 2$ vertices to ensure for each $x \in \inte(F_1)$ either $|N_G(x)\cap~\hspace{-0.15cm}B'|~=~0$ or $|N_G(x)\cap B'| \geq 2$.} The new set $B'$ satisfies $|B'| \geq |B|/2$. For each $v\in B'$, let $P_v\in \cP$ be the path with $v$ as an endvertex. 
For each $v\in B'$, remove any edges between $v$ and $V(P_v)$ in $G$, excluding the edge emanating from $v$ in $P_v$. Call the resulting graph $G'$.
Note that we have removed at most $|B'|\,m$ edges. 
There are thus at least $|B'|\cdot (d/2)-|B'|\,m\geq |B'|\,d/8$ edges from $B'$ to $V(G)\setminus (B'\cup W')$ in $G'$. (Note that there are no edges inside~$B'$ as our graph is bipartite.)
% Recall that we deleted any edges between $v\in A'$ and $P_v$ from $G$ when forming $G'$.
Now, by construction of $G'$ and that no vertex $w\notin W'$ has at least two neighbours in $B'$ in $G'$, we get $|N_{G'}(B',V(G)\setminus W')|\geq |B'|\,d/8\geq |B|\,d/16= d^2m^{25}/16$. 

Let $C:=N_{G'}(B',V(G)\setminus W')$.
Let $W'':= W'\cup (\cup_{v\in B'}V(P_v))$, noting at this point that $|W''|\leq 2dm^{11}+dm^{25}\cdot m\leq 2dm^{26}$. 
Next, we apply Lemma~\ref{lem-mpath} to connect $C$ and $\bd(F_3)$, the former of size at least $d^2m^{25}/16$ and the latter of size at least $2dm^{28}$, with a path $P$ of length at most $m$ avoiding $W''$, as $d \geq m^{200}$. 
Let $u\in N_{G'}(B')$ be the endvertex of $P$ in $C$,  $w_1$ be a neighbour of $u$ in $B'$ and $x$ be the neighbour of $w_1$ in $\inte(F_1)$. 
Let $Q_3$ be the path in $F_3$ between $v_3$ and $P \cap \bd(F_3)$.
Fix another path $P' \in \mathcal{P}$ which has an endvertex $w_2$ in $(\bd(F_1) \cap N(x))\setminus \{w_1\}$. That is, $P' = P_{w_2}$.
As $P$ avoids $W''$, $P$ is disjoint from $P_{w_1}$ and $P_{w_2}$.
%We deleted any edges between $v$ and $P_v$ from $G$ when forming $G'$. Thus, we have that $u\notin V(P_v)$. Thus, as $P$ has no internal vertices in $W''$, $P$ and $P_v$ are vertex-disjoint.
%
%Now, $\cP$ is a set of $dm^{14}$ vertex disjoint paths, so certainly there is some $P'\in \cP\setminus \{P_v\}$ noting containing any vertex in $P$. 
Let $Q_4,Q_5$ be the paths in $F_2$ joining $P_{w_1}$ and $P_{w_2}$, respectively, with $v_2$.
Then
	$v_3Q_3Pw_1P_{w_1}Q_4v_2$ and
	$v_3Q_3Pw_1xw_2P_{w_2}Q_5v_2$
	are two $v_3,v_2$-paths whose lengths differ by two and are at most 
	$$
	e(Q_3)+e(P)+3+e(P_{w_2})+e(Q_5)\le 20m+|P|+|P_{w_2}|+5\le 50m+2.
	$$
	Observe that $|(V(P_{w_1}) \cup V(P_{w_2})) \cap \bd(F_3)| \leq 2m \leq \sqrt{d}m^5$. Hence, we can construct an $(3\sqrt{d}m^6,\sqrt{d}m^{22},10m)$-unit $F_3'$, with core vertex $v_3$, from $F_3$ by removing all $\{v_3\}, (V(P_{w_1}) \cup V(P_{w_2})) \cap \bd(F_3)$-paths (in $F_3$) and leaves of the stars incident to these paths, as well as removing the path $Q_3$ and the leaves of the star incident to $Q_1$. Let $F_2'$ be the $(3\sqrt{d}m^6,\sqrt{d}m^{22},10m)$-unit with core vertex $v_2$ constructed from $F_2$ by removing the paths $Q_4$ and $Q_5$ and the leaves of the stars incident to $Q_4$ and $Q_5$. (If $Q_4$ and $Q_5$ only differ by an edge, then remove in addition an arbitrary path and its leaves from $F_2$ to construct $F_2'$.) Let 
	$$
	A^*:=(V(P) \cup V(P_{w_1}) \cup \{x\} \cup V(P_{w_2})\cup V(Q_3)\cup V(Q_4)\cup V(Q_5))
	\setminus\{v_2,v_3\}.
	$$
	We clearly have $|A^*|\le 2\cdot 50m$.
	Therefore, 
	$$
	P^* = (v_3,v_2,F_3',F_2',3,A^*)
	$$ 
	is a $(50m,1)$-adjuster in $G-W$. Also, it is routine to check that $|A^*|\le 2\ell(P^*)$, finishing the proof of the lemma.
%	By construction, we also have $|A^*| \leq 2\ell(P^*)$, as required.
\end{proof}

\definecolor{applegreen}{rgb}{0.55, 0.71, 0.0}
\definecolor{darkspringgreen}{rgb}{0.09, 0.45, 0.27}
\definecolor{darktangerine}{rgb}{1.0, 0.66, 0.07}
\definecolor{darkcyan}{rgb}{0.0, 0.55, 0.55}
\definecolor{darkcoral}{rgb}{0.8, 0.36, 0.27}
\definecolor{babyblue}{rgb}{0.54, 0.81, 0.94}
\definecolor{darkcerulean}{rgb}{0.03, 0.27, 0.49}

\begin{figure}[!ht]
	\centering
	\begin{tikzpicture}[scale=1.4]
	%unit1
	\node[inner sep= 1pt](v1) at (0,3)[circle,fill]{};
	\node[inner sep= 1pt](v1l) at (-0.3,3)[]{\small $v_1$};
	\node[inner sep= 1pt](v11) at (2,4)[circle,fill]{};
	\node[inner sep= 1pt,red](v12) at (2,3)[circle,fill]{};
	\node[inner sep= 1pt,black](v13) at (2,2)[circle,fill]{};
	\node[inner sep= 1pt,red](xl) at (1.9,3.2)[]{\small $x$};
	%\node[inner sep= 1pt,darkpastelgreen](q2l) at (1.5,2)[]{\small $Q_2$};
	\draw[decorate, decoration=snake, segment length=6mm,black] (v1) -- (v11);
	\draw[decorate, decoration=snake, segment length=6mm,black] (v1) -- (v12);
	\draw[decorate, decoration=snake, segment length=6mm,black] (v1) -- (v13);
	\node[inner sep= 1pt](a1) at (2.8,4.3)[circle,fill]{};
	\node[inner sep= 1pt](a2) at (2.8,4.1)[circle,fill]{};
	\node[inner sep= 1pt](a3) at (2.8,3.9)[circle,fill]{};
	\node[inner sep= 1pt](a4) at (2.8,3.7)[circle,fill]{};
	\draw (v11) -- (a1);
	\draw (v11) -- (a2);
	\draw (v11) -- (a3);
	\draw (v11) -- (a4);
	\node[inner sep= 1pt,red](b1) at (2.8,3.3)[circle,fill]{};
	\node[inner sep= 1pt,red](b1l) at (2.8,3.5)[]{\small $w_2$};
	\node[inner sep= 1pt](b2) at (2.8,3.1)[circle,fill]{};
	\node[inner sep= 1pt](b3) at (2.8,2.9)[circle,fill]{};
	\node[inner sep= 1pt,red](b4) at (2.8,2.7)[circle,fill]{};
	\node[inner sep= 1pt,red](b4l) at (2.8,2.5)[]{\small $w_1$};
	\draw (v12) -- (b1);
	\draw (v12) -- (b2);
	\draw (v12) -- (b3);
	\draw (v12) -- (b4);
	\node[inner sep= 1pt](c1) at (2.8,2.3)[circle,fill]{};
	\node[inner sep= 1pt,black](c2) at (2.8,2.1)[circle,fill]{};
	%\node[inner sep= 1pt,deepcarmine](c2l) at (3.05,2.05)[]{\small $w_2$};
	\node[inner sep= 1pt](c3) at (2.8,1.9)[circle,fill]{};
	\node[inner sep= 1pt](c4) at (2.8,1.7)[circle,fill]{};
	\draw (v13) -- (c1);
	\draw (v13) -- (c2);
	\draw (v13) -- (c3);
	\draw (v13) -- (c4);
	\draw[rotate=90][dashed] (3,-2.8) ellipse (50pt and 15pt);
	\draw[rotate=90,cadmiumgreen,thick] (2.6,-2.8) ellipse (35pt and 12pt);
	\node[inner sep= 1pt,cadmiumgreen](Al) at (2.4,3.6)[]{\small $B$};
	\draw[rotate=90,darkslateblue,thick] (2.7,-2.8) ellipse (20pt and 10pt);
	\node[inner sep= 1pt,darkslateblue](Aal) at (2.1,2.6)[]{\small $B'$};
	%unit2	
	\node[inner sep= 1pt](v2) at (9,3)[circle,fill]{};
	\node[inner sep= 1pt](v2l) at (9.3,3)[]{\small $v_2$};
	\node[inner sep= 1pt,darkpastelgreen](v21) at (7,4)[circle,fill]{};
	\node[inner sep= 1pt,darkpastelgreen](v22) at (7,3)[circle,fill]{};
	\node[inner sep= 1pt](v23) at (7,2)[circle,fill]{};
	%\node[inner sep= 1pt,darkpastelgreen](r1l) at (7.5,4.1)[]{\small $R_1$};
	%\node[inner sep= 1pt,darkpastelgreen](r2l) at (7.5,3.25)[]{\small $R_2$};
	\draw[decorate, decoration=snake, segment length=6mm,darkpastelgreen] (v2) -- (v21);
	\draw[decorate, decoration=snake, segment length=6mm,darkpastelgreen] (v2) -- (v22);
	\draw[decorate, decoration=snake, segment length=6mm,black] (v2) -- (v23);
	\node[inner sep= 1pt](d1) at (6.2,4.3)[circle,fill]{};
	\node[inner sep= 1pt](d2) at (6.2,4.1)[circle,fill]{};
	\node[inner sep= 1pt](d3) at (6.2,3.9)[circle,fill]{};
	\node[inner sep= 1pt](d4) at (6.2,3.7)[circle,fill]{};
	\draw (v21) -- (d1);
	\draw (v21) -- (d2);
	\draw[darkpastelgreen] (v21) -- (d3);
	\draw (v21) -- (d4);
	\node[inner sep= 1pt](e1) at (6.2,3.3)[circle,fill]{};
	\node[inner sep= 1pt](e2) at (6.2,3.1)[circle,fill]{};
	\node[inner sep= 1pt](e3) at (6.2,2.9)[circle,fill]{};
	\node[inner sep= 1pt](e4) at (6.2,2.7)[circle,fill]{};
	\draw (v22) -- (e1);
	\draw (v22) -- (e2);
	\draw[darkpastelgreen] (v22) -- (e3);
	\draw (v22) -- (e4);
	\node[inner sep= 1pt](f1) at (6.2,2.3)[circle,fill]{};
	\node[inner sep= 1pt](f2) at (6.2,2.1)[circle,fill]{};
	\node[inner sep= 1pt](f3) at (6.2,1.9)[circle,fill]{};
	\node[inner sep= 1pt](f4) at (6.2,1.7)[circle,fill]{};
	\draw (v23) -- (f1);
	\draw (v23) -- (f2);
	\draw (v23) -- (f3);
	\draw (v23) -- (f4);
	\draw[rotate=90][dashed] (3,-6.2) ellipse (50pt and 15pt);
	%paths joining boundaries of units
	\draw[decorate, decoration=snake, segment length=10mm,babyblue] (a4) -- (d2);
	\draw[decorate, decoration=snake, segment length=10mm,red,thick] (b1) -- (d3);
	\draw[decorate, decoration=snake, segment length=10mm,babyblue] (b2) -- (d4);
	\draw[decorate, decoration=snake, segment length=10mm,babyblue] (b3) -- (e1);
	\draw[decorate, decoration=snake, segment length=10mm,red,thick] (b4) -- (e3);
	\draw[decorate, decoration=snake, segment length=10mm,babyblue] (c1) -- (e4);
	\draw[decorate, decoration=snake, segment length=10mm,babyblue] (c2) -- (f1);
	\draw[decorate, decoration=snake, segment length=10mm,babyblue] (c3) -- (f2);
	%\draw[decorate, decoration=snake, segment length=10mm,celestialblue] (c4) -- (f3);
%	\node[inner sep= 1pt,celestialblue](P) at (4.5,4.2)[]{\small $\mathcal{P}$};
	%\node[inner sep= 1pt,red](w) at (3.5,2.6)[circle,fill]{};
	%\node[inner sep= 1pt,red](wl) at (3.7,2.6)[]{\small $w$};
	%\draw[red] (w) -- (b4);
	%\draw[red] (w) -- (c2);
	%\node[inner sep= 1pt,deepcarmine](P1l) at (4.5,3.5)[]{\small $P_1$};
	%\node[inner sep= 1pt,deepcarmine](P2l) at (4,2.25)[]{\small $P_2$};
	\node[inner sep= 1pt,black](bd1l) at (2.8,5)[]{\small $\bd(F_1)$};
	\node[inner sep= 1pt,black](bd2l) at (6.2,5)[]{\small $\bd(F_2)$};
	%unit3
	\node[inner sep= 1pt](v3) at (4.5,-2)[circle,fill]{};
	\node[inner sep= 1pt](v3l) at (4.5,-2.2)[]{\small $v_3$};
	\node[inner sep= 1pt](v31) at (3,-1)[circle,fill]{};
	\node[inner sep= 1pt,darkpastelgreen](v32) at (4.5,-1)[circle,fill]{};
	\node[inner sep= 1pt,black](v33) at (6,-1)[circle,fill]{};
	\draw[decorate, decoration=snake, segment length=6mm,black] (v3) -- (v31);
	\draw[decorate, decoration=snake, segment length=6mm,darkpastelgreen] (v3) -- (v32);
	\draw[decorate, decoration=snake, segment length=6mm,black] (v3) -- (v33);
	\node[inner sep= 1pt](g1) at (6.3,-0.5)[circle,fill]{};
	\node[inner sep= 1pt](g2) at (6.1,-0.5)[circle,fill]{};
	\node[inner sep= 1pt](g3) at (5.9,-0.5)[circle,fill]{};
	\node[inner sep= 1pt](g4) at (5.7,-0.5)[circle,fill]{};
	\draw (v33) -- (g1);
	\draw (v33) -- (g2);
	\draw (v33) -- (g3);
	\draw (v33) -- (g4);
	\node[inner sep= 1pt](h1) at (4.8,-0.5)[circle,fill]{};
	\node[inner sep= 1pt](h2) at (4.6,-0.5)[circle,fill]{};
	\node[inner sep= 1pt](h3) at (4.4,-0.5)[circle,fill]{};
	\node[inner sep= 1pt,darkpastelgreen](h4) at (4.2,-0.5)[circle,fill]{};
	\draw (v32) -- (h1);
	\draw (v32) -- (h2);
	\draw (v32) -- (h3);
	\draw[darkpastelgreen] (v32) -- (h4);
	\node[inner sep= 1pt](i1) at (3.3,-0.5)[circle,fill]{};
	\node[inner sep= 1pt](i2) at (3.1,-0.5)[circle,fill]{};
	\node[inner sep= 1pt](i3) at (2.9,-0.5)[circle,fill]{};
	\node[inner sep= 1pt](i4) at (2.7,-0.5)[circle,fill]{};
	\draw (v31) -- (i1);
	\draw (v31) -- (i2);
	\draw (v31) -- (i3);
	\draw (v31) -- (i4);
	\draw[darktangerine,rotate=45] (3.8,-2) ellipse (25pt and 10pt);
	\draw[dashed,darktangerine] (2.8,2) -- (3.4,0.8);
	\draw[dashed,darktangerine] (2.8,3.4) -- (4.8,1.86);
	\node[inner sep= 1pt,applegreen](u) at (4,1)[circle,fill]{};
	\draw[darkcerulean] (u) -- (b4);
	\draw[decorate, decoration=snake, segment length=10mm,applegreen] (u) -- (h4);
	\node[inner sep= 1pt](v3l) at (3.8,1)[]{\small $u$};
	\node[inner sep= 1pt,darkpastelgreen](Q3l) at (4.8,-1.3)[]{\small $Q_3$};
	\node[inner sep= 1pt,darkpastelgreen](Q4l) at (8,3.8)[]{\small $Q_5$};
	\node[inner sep= 1pt,darkpastelgreen](Q5l) at (8,3.2)[]{\small $Q_4$};
	\node[inner sep= 1pt,darktangerine](Bl) at (4.8,1.3)[]{\small $C$};
	\node[inner sep= 1pt,applegreen](Pl) at (4.3,0.3)[]{\small $P$};
	\node[inner sep= 1pt,red](pw2l) at (4,3.65)[]{\small $P_{w_2}$};
	\node[inner sep= 1pt,red](pw1l) at (4.7,2.95)[]{\small $P_{w_1}$};
	\draw[dashed] (4.5,-0.5) ellipse (60pt and 12pt);
	\node[inner sep= 1pt](bd3l) at (7.1,-0.5)[]{\small $\bd(F_3)$};
	\end{tikzpicture}
	\caption{An illustration of the second case in the proof of Lemma~\ref{lem-router-dense-max}.}\label{fig-build-units-v1-v3}
\end{figure}
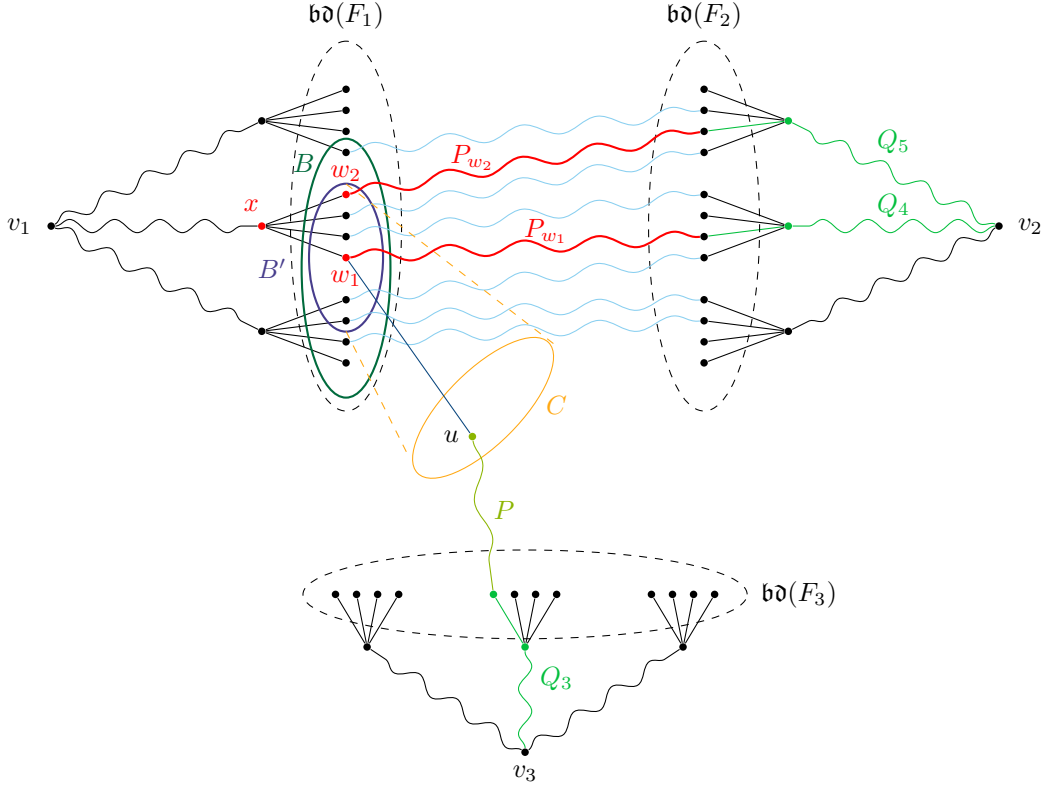

\subsection{Constructing larger adjusters}

\begin{lemma}\label{lem-chain-dense-max}
For each $0<\ep_1,\ep_2<1$,  the following holds for all sufficiently large~$K$.
For $d > 0$, let $G$ be an $n$-vertex bipartite $(\ep_1,\ep_2d)$-expander with $\delta(G)\ge d$ and $n\ge Kd$.
Suppose $d\ge m^{200}$ and $G$ contains no copy of  $\tk_{\sqrt{d}m}^{(1)}$. 
Let $W\subseteq V(G)$ satisfy $|W|\leq dm^{10}$.
Then, there exists an $(100m^3,m^2)$-adjuster $P = (v_1, v_2, F_1, F_2,2, A)$ in $G-W$ such that $80m^3 \leq \ell(P) \leq 80m^3 + 80m$.
\end{lemma}

\begin{proof}
Inductively on $i=1,2,\dots$
    we will construct, in the graph $G-W$, a $(80mi,i)$-adjuster $P_i= (v^i_1, v^i_2, F^i_1, F^i_2,z_i, A_i, \mathcal{P}_i)$ with the additional properties that $z_i\ge 2$ and $|A_i|\le 2\ell(P_i)$, stopping when its length $\ell(P_i)$ becomes at least~$80m^3$ for the first time.
    
For $i=1$, we apply Lemma~\ref{lem-router-dense-max} in order to get a $(50m, 1)$-adjuster $P_1 = (v^1_1, v^1_2, F^1_1, F^1_2,3, A_1, \mathcal{P}_1)$ in $G-W$, which can be done since $|W|\leq dm^{10}\leq dm^{11}$.
    
    We then proceed as follows. Suppose that we have a $(80m(i-1),i-1)$-adjuster 
    $$P_{i-1} = (v^{i-1}_1, v^{i-1}_2, F_1^{i-1}, F_2^{i-1},z_{i-1}, A_{i-1},  \mathcal{P}_{i-1})
    $$ 
    as above, whose length $\ell(P_i)$ is still less than~$80m^3$. 
    %It will be clear from construction that $z_{i} \in [5/2,3]$. 
    Define $W_{i} := W  \cup p(P_{i-1})$. Then
    $$|W_{i}| = |W \cup A_{i-1}\cup \inte(F_1^{i-1})\cup\inte(F_2^{i-1})| \leq dm^{10}+ 2\cdot 80m(i-1)+ 6\sqrt{d}m^6\cdot 10m + 2 \leq dm^{11},$$
     since $i< \ell(P_i)<  80m^3$.
   We then apply Lemma~\ref{lem-router-dense-max} to get a $(50m, 1)$-adjuster $\widetilde{P}=(w_1,w_2,\widetilde{F}_1,\widetilde{F}_2,3,\widetilde{A})$ in $G-W_i$. Define $W_i' := W_i \cup p(\widetilde{P})$. Clearly, $|W_i'| \leq dm^{11}$ also.
    
    %Let $F_1,F_2$ be the units of the adjuster $P$ with core vertices $w_1,w_2$, respectively and
    %, and let $X$ be the set of vertices in $\bd(F_1)$ such that the paths between $w_1$ and the vertices of $\inte(F_1)$ neighbouring vertices of $X$ are disjoint from $W_i$. 
    %More specifically, since these paths lie in $\inte(F_1)$, they are disjoint from $W$ (by definition of $P$ being an adjuster in $G - W_i$). %, therefore, they are disjoint from $V(P_{i-1})$. 
    %Since there are $\sqrt{d}m^6$ different paths in $\inte(F_1)$, $d \geq m^{200}$ and $|V(P_{i-1})| \leq 200m^3$, we have that $$|X|\ge(\sqrt{d}m^6-|V(P_{i-1})|)\cdot \sqrt{d}m^{21}\ge dm^{15}.$$
    Let $X:=\bd(\widetilde{F}_1)$ and $Y:=\bd(F^{i-1}_2)\setminus p(\widetilde{P})$. % If $N_{G-W_i}(w_1) \cap N_{G-W_i}(v^{i-1}_1) \neq \varnothing$, then linking $w_1$ and $v^{i-1}_1$ via a common neighbour creates a $(100mi, i)$-adjuster $P_i$. 
   Apply Lemma~\ref{lem-mpath} with $(A, B, x)_{\ref{lem-mpath}} = (X, Y, dm^{28})$ in order to find an $X,Y$-path of length at most $m$ in $G - W_i'$. 
   %OP: the following is not necessary:
   %We can assume this $X,Y$-path intersects with $\bd(F_1^{i-1}) \cup \bd(F_2^{i-1})$ in precisely one vertex, swapping indices of units if necessary.
   Thus there exists a $v^{i-1}_2,w_1$-path $Q$ of length at most $10m + m + 10m < 22m$. Consider paths between $v_1^{i-1}$ and $w_2$ obtained by first taking a $v_1^{i-1},v_2^{i-1}$-path from the adjuster $P_{i-1}$, followed by $Q$, followed by a $w_1,w_2$-path from the adjuster~$\widetilde{P}$. Clearly, we can choose $i+1$ of these paths so that their lengths form an arithmetic progression with difference $2$ and are all at most $\ell(P_{i-1})+e(Q)+\ell(\widetilde{P})$. Let $A_i$ be the set of the vertices used by these paths, except for their endpoints $v_1^{i-1}$ and $w_2$. The longest among these paths has length
   \begin{equation}\label{eq-length-of-P_i}
   \ell_i:=\ell(P_{i-1})+e(Q)+\ell(\widetilde{P})\le \ell(P_{i-1})+80m,
   \end{equation}
   where  by induction the right-hand side is at most $(80m(i-1)+2(i-1))+80m\le 80mi+2i$.
   Again by induction, we have
   $$
   |A_i|\le |A_{i-1}|+|Q|+|\widetilde{A}| \le 2\cdot 80m(i-1)+22m+2\cdot 50m\le 2\cdot 80mi.
   $$
   Truncate our units $F_1^{i-1}$ and $\widetilde{F}_2$ in order to ensure that $A_i$ is disjoint from $V(F_1^{i-1}) \cup V(\widetilde{F}_2)$. Since $i < 80m^3$ and $n/d \geq K$, this requires removing on total at most $|A_i| \leq 2\cdot 80mi \leq m^5$ stars\vspace{0.1cm} from each original unit.\vspace{0.1cm}
   It follows  that, for some $z_{i} \in [2,3]$, $P_i:=(v_1^{i-1},w_2,F_1^{i-1}, \widetilde{F}_2^i,z_i ,A_i)$ is a $(80mi, i)$-adjuster with $\ell(P_i)=\ell_i$. Furthermore, by induction and our choice of $\widetilde{P}$, we have\vspace{0.08cm} $|A_i|\le 2\ell(P_{i-1})+|e(Q)|+2\ell(\widetilde{P})\le 2\ell(P_i)$, as claimed.
   %\footnote{OP: When we pass from $P_{m^2}$ to $P$ we add more vertices to $A$ but keep the parameter $100m^3$ unchanged. Add an argument that the gap $100m^3-|A_{m^2}|$ is large enough.}
    
    Thus we can always proceed until we reach an $(80mi,i)$-adjuster $P_i=(v_1^i,v_2^i,F_1^i,F_2^i,z_i,A_i,\mathcal{P}_i)$ such that, in addition to $z_i\ge 2$ and $|A_i|\le 2\ell(P_i)$, we have $\ell(P_i)\ge 80m^3$. By~\eqref{eq-length-of-P_i}  we increase the length at each stage by at least $1$ and at most $80m$, so
    we have that $\ell(P_i)\le 80m^3 + 80m$ and $i\ge 80m^3/80m=m^2$.  Moreover, $|A_i|/2$ does not  exceed $\ell(P_i) \leq 80m^3 + 80m \leq 100m^3$ so $P_i$ is also a $(100m^3,i)$-adjuster. Finally, by taking only the shortest $m^2+1$ paths of $\mathcal{P}_i$ (and trimming the units $F_1^i$ and $F_2^i$ to make $z_i$ exactly 2), we get a $(100m^3,m^2)$-adjuster with all required properties.
\end{proof}

\section{Using the adjusters}\label{sec-link-adj}

\begin{lemma}\label{lem-precise-path}
     For each $0<\ep_1,\ep_2<1$,  the following holds for all sufficiently large~$K$.
     For $d > 0$, let $G$ be an $n$-vertex bipartite $(\ep_1,\ep_2d)$-expander with $\delta(G)\ge d$ and $n\ge Kd$. 
     Suppose $d\ge m^{200}$ and $G$ contains no copy of $\tk_{\sqrt{d}m}^{(1)}$.
     Let $W \subseteq V(G)$ be such that $|W| \leq dm^{9}$. Let $F_1,F_2$ be two $(\sqrt{d}m^6+1,\sqrt{d}m^{22},10m)$-units, with core vertices $v_1,v_2\in V(G)\setminus W$ which lie on the same part of the bipartite graph~$G$.
     Further, suppose that $|(\inte(F_{1})\cup\inte(F_2))\cap W|\le \sqrt{d}m^5$. 
     Then there exists a $v_1, v_2$-path $L$ of length precisely $80m^3$ in $G - W$.% such that $|(\inte(F_{1})\cup\inte(F_2))\cap L|\le m^2$.
\end{lemma}

\begin{proof}
   Let $W' := W \cup \inte(F_1) \cup \inte(F_2)$. As $|W'| \leq dm^{10}/3$, we can apply Lemma~\ref{lem-chain-dense-max} to find a $(100m^3,m^2)$-adjuster $P=(w_1,w_2,E_1,E_2, 2,A)$ in $G-W'$ such that $80m^3 \leq \ell(P) \leq 80m^3 + 80m$.
   
   Let $X$ be the set of vertices $v$ in $\bd(F_1)$ such that the path between $v_1$ and $v$ is internally vertex-disjoint from $W$. 
   There are $\sqrt{d}m^6$ different paths in $\inte(F_1)$, thus $$|X|\ge(\sqrt{d}m^6-\sqrt{d}m^5)\cdot \sqrt{d}m^{22}\ge dm^{15}.$$
   
    Let $Y:=\bd(E_1)$ and $W'':=W'\cup p(P)$. We have $|W''|\le dm^{10}/2$. Hence, we can apply Lemma~\ref{lem-mpath} with $( A, B, x) = (X, Y, dm^{12})$ in order to find an $X,Y$-path $Q_1$ of length at most $m$ connecting $\bd(F_1)$ and $\bd(E_1)$ in $G - W''$. 
    Let $W''' := W'' \cup V(Q_1)$ and observe that $|W'''| \leq dm^{10}$.
    Then apply Lemma~\ref{lem-mpath} similarly as before to find a path $Q_2$ of length at most $m$ connecting $\bd(F_2)$ and $\bd(E_2)$ in $G - W'''$.
   
   Let $R_1$ be the path from $v_1$ to $Q_1 \cap \bd(F_1)$, $S_1$ be the path from $Q_1 \cap \bd(E_1)$ to $w_1$, $S_2$ be the path from $w_2$ to $Q_2 \cap \bd(E_2)$ and $R_2$ be the path from $Q_2 \cap \bd(F_2)$ to $v_2$, with all these paths taken inside the respective units. 
   Observe that $|R_i| \leq 10m + 1$ and $|S_i|\leq 10m+1$ for $i = 1,2$. Removing the paths $R_1$ and $R_2$, and incident leaves, from $F_1$ and $F_2$, respectively, we see that 
   $$P':=(v_1,v_2,F_1,F_2, R_1 \cup Q_1 \cup S_1 \cup A \cup S_2 \cup Q_2 \cup R_2)$$
   is a $(101m^3,m^2)$-adjuster in $G-W$.
   Furthermore, we have $$80m^3 \leq \ell(P) \le \ell(P')\le \ell(P)+100m\le 80m^3+200m.$$ 
   Since $v_1$ and $v_2$ lie on the same part of $G$, $\ell(P')$ is even. 
   Thus, since $P'$ is a $(101m^3,m^2)$-adjuster and $m^2 \geq 200m$ (as $n \geq Kd$), we can find a $v_1, v_2$-path $L$ of length precisely $80m^3$ in $G - W$.% and by construction $|(\inte(F_{1})\cup\inte(F_2))\cap L|\le 20m + 2\le m^2$.
\end{proof}

We can now prove our main result, Theorem~\ref{main-theorem}.

\begin{proof}\label{proofofmaintheorem}    
Set $t:=\sqrt{d}m$. Using Lemma~\ref{lem-unit} iteratively, we choose $2t$ $(\sqrt{d}m^6+1,\sqrt{d}m^{22},10m)$-units  such that their interiors are pairwise vertex disjoint. This is possible since each interior has at most $\sqrt{d}m^6\cdot 10m$ vertices, which is smaller than $dm^{30}/(2t)$.
    Clearly, we can choose some $t$ of these units $F_1,F_2,\cdots, F_t$ such that their core vertices $v_1,v_2,\dots,v_t$ lie in the same part of the bipartite graph~$G$. 
    
    We have to show that $G$ contains a copy of $\tk_{t}^{(1)}$ or $\tk_{t}^{(\ell)}$, where $\ell := 80m^3$.
    To this end, assume $G$ is $\tk_{t}^{(1)}$-free. Let $\mathcal{P} := \{P_1, \ldots, P_S\}$ be a maximal collection of internally vertex disjoint paths such that:
    
    \begin{itemize}
        \item for each $s \in [S]$, $P_s$ is a $v_i, v_j$-path of length precisely $\ell$ for some distinct $i,j \in [t]$;
        
        \item if $P_s$ is a $v_i, v_j$-path, then for every $k\in[t]\setminus\{i,j\}$, $P_s$ is disjoint from $\inte(F_k)$;%then $|(\inte(F_i)\cup\inte(F_j))\cap P_s|\le m^2$, and
        
        \item for distinct $i,j \in [t]$, there is at most one path in $\mathcal{P}$ with $v_i$ and $v_j$ as end vertices.
    \end{itemize}
    
    If $S = \binom{t}{2}$, then the graph formed by all the paths in $\mathcal{P}$ is our desired copy of $\tk_{t}^{(\ell)}$.
    Hence, we may assume that there exist (distinct) $i,j \in [t]$ such that $\mathcal{P}$ contains no  $v_i, v_j$-path.
    
    Let $W := (\cup_{s \in [S]}V(P_s)\setminus\{v_i, v_j\}) \cup \{\inte(F_q): q\in [t]\setminus\{i, j\}\}$. Then $$|W| \leq t^2\ell + t\cdot\sqrt{d}m^6\cdot 10m \le dm^{9}.$$ Furthermore, 
    \[
    |\inte(F_i)\cap W|=|\inte(F_i)\cap (\cup_{s \in [S]}V(P_s)\setminus\{v_i\})|\le \sqrt{d}m\cdot 80m^3\le \frac{1}{2}\sqrt{d}m^5,
    \] 
    and similarly $|\inte(F_j)\cap W|\leq \frac{1}{2}\sqrt{d}m^5.$ Thus, by Lemma~\ref{lem-precise-path} there exists a $v_i,v_j$-path $P_{S+1}$ of length $\ell$ in $G-W$. %with $|(\inte(F_i)\cup\inte(F_j))\cap P_{S+1}|\le m^2$.
     Also, since $P_{S+1}$ is in $G-W$, this path is disjoint from $\inte(F_k)$ for $k\neq i,j$ and\vspace{0.05cm} internally disjoint from all paths in~$\mathcal P$.
     This contradicts the maximality of $\mathcal{P}$.
    Thus $S = \binom{t}{2}$ and we are done.
\end{proof}

\section{Concluding Remarks}\label{sec-conclude}
While our main result gives the optimal degree bound forcing balanced clique subdivisions, it would be very interesting to consider balanced subdivisions of more general graphs. Let $H$ be a graph with $p$ vertices and $q$ edges. It is known that if $G$ is a graph with average degree at least $88(p + q)$, then $G$ contains a subdivision of $H$. To see this, we need a piece of notation. A graph is \emph{$k$-linked} if, for any choices $x_1,...,x_k,y_1,...,y_k$ of $2k$ distinct vertices there are vertex disjoint paths $P_1,...,P_k$ with $P_i$ joinings $x_i$ to $y_i$, for all $i\in [k]$. We first find a subgraph $G'$ in $G$ which is $22(p+q)$-connected due to a result of Mader~\cite{Mader72averagedegree}. Then $G'$ contains a subgraph $G''$ which is $(p+q)$-linked due to a result of Bollob\'as-Thomason~\cite{boltomhighly}. One can then embed an $H$-subdivision in $G''$ by taking $\{x_i,y_i\}_{i\in [q]}$ to be the pair of endvertices of edges of $H$.

Perhaps the following is true.

\begin{prob}
    Does there exist a constant $C$ such that for any $p$-vertex $q$-edge graph $H$, if $G$ has average degree at least $C(p+q)$ then $G$ contains a balanced subdivision of $H$?
\end{prob}

\paragraph{Note added before submission.}
While preparing this paper, we learnt that Bingyu Ruan, Yantao Tang, Guanghui Wang and Donglei Yang independently proved Theorem~\ref{thm-main}.

\bibliographystyle{plain}
\bibliography{bib-balanced-subdiv}

\end{document}